\documentclass[12pt]{article}

\usepackage{amsmath,amsfonts,amssymb, enumerate,  MnSymbol}
\usepackage{mathrsfs}  
\usepackage{graphicx}
\usepackage{subfig}
\usepackage[numbers, sort]{natbib}
\usepackage{hyperref}
\usepackage{stmaryrd}
\usepackage{xy}
\usepackage{mathtools}
\usepackage{xcolor}
\usepackage{tikz}
\usepackage{manfnt}
\usepackage{ntheorem}
\usepackage{accents}

\usetikzlibrary{topaths,calc}

\input xy
\xyoption{all}

\numberwithin{equation}{section}

\newtheorem{theorem}{Theorem}[section]

\newtheorem{proposition}[theorem]{Proposition}
\newtheorem{lemma}[theorem]{Lemma}
\newtheorem{corollary}[theorem]{Corollary}
\newtheorem{definition}{Definition}[section]
\newtheorem{exmp}{Example}[section]

\newenvironment{proof}[1][Proof]{\begin{trivlist}
\item[\hskip \labelsep {\bfseries #1}]}{\end{trivlist}}
\newenvironment{example}[1][Example]{\begin{trivlist}
\item[\hskip \labelsep {\bfseries #1}]}{\end{trivlist}}
\newenvironment{remark}[1][Remark]{\begin{trivlist}
\item[\hskip \labelsep {\bfseries #1}]}{\end{trivlist}}

\newcommand{\qed}{\nobreak \ifvmode \relax \else
      \ifdim\lastskip<1.5em \hskip-\lastskip
      \hskip1.5em plus0em minus0.5em \fi \nobreak
      \vrule height0.75em width0.5em depth0.25em\fi}

\begin{document}

\title{Inductive construction of path homology
chains}
\author{Matthew Burfitt and Tyrone Cutler}
\date{}
\maketitle

\begin{abstract}
   Path homology plays a central role in digraph topology and GLMY theory more general. Unfortunately, the computation of the path homology of a digraph $G$ is a two-step process, and until now no complete description of even the underlying chain complex has appeared in the literature.
   
    In this paper we introduce an inductive method of constructing elements of the path homology chain modules $\Omega_n(G;R)$ from elements in the proceeding two dimensions. This proceeds via the formation of what we call upper and lower \emph{extensions}, that are parametrised by certain labeled multihypergraphs which we introduce and call \emph{face multihypergraphs}.

    When the coefficient ring $R$ is a finite field the inductive elements we construct generate $\Omega_*(G;R)$. With integral or rational coefficients, the inductive elements generate at least $\Omega_i(G;R)$ for $i=0,1,2,3$. Since in low dimensions the inductive elements extended over labeled multigraphs coincide with naturally occurring generating sets up to sign, they are excellent candidates to reduce to a basis.

    Inductive elements provide a new concrete structure on the path chain complex that can be directly applied to understand path homology, under no restriction on the digraph $G$.
    We employ inductive elements to construct a sequence of digraphs whose path Euler characteristic can differ arbitrarily depending on the choice of field coefficients.
    In particular, answering an open question posed by Fu and Ivanov.
\end{abstract}

\tableofcontents

\section{Introduction}

Path homology, one of the several proposed homology theories for directed graphs $G$, is based around the idea of detecting cycles in equal length paths within a graphical structure. Its construction was laid out by Grigor'yan, Lin, Muranov, and Yau in the foundation paper \cite{Grigoryan2013}, which built on earlier work of Dimakis and M\"{u}ller-Hoissen \cite{Dimakis1994}, where elements of the dual cohomology theory were considered in the study of field theories on discrete space time.

Path homology now sits in the broader study of geometry, topology and homotopy theory of digraphs and quivers often referred to as GLMY-theory.

A primary motivation for path homology lies in the fact that the simplicial homology of any abstract simplicial complex can be realised as the path homology of a digraph which is obtained as the directed edges of a barycentric subdivision of the complex \cite{Grigor'yan2014}.
Moreover, path homology satisfies certain Eilenberg–Steenrod axioms \cite{Grigoryan2018}, and shares variations of many of the important properties enjoyed by singular and simplicial homology, including functorality, homotopy invariance \cite{Grigor'yan2014b}, and a K\"{u}nneth Theorem for the box product of digraphs \cite{Grigoryan2017}. In particular, these fundamental properties are not necessarily shared by other homology theories of digraphs.
Moreover, path homology cannot in general be obtained as the homology of a space \cite{Fu2024}, unlike those homologies derived through singular mappings such as directed cliques.

Furthermore, path homology appears as a graded submodule of the bigraded path homology \cite{Hepworth2024}, emerging as the second page of a spectral sequence whose first page is the naturally bigraded magnitude homology \cite{Hepworth2017}. Although the spectral sequence was introduced in \cite[Remark 8.7]{Hepworth2017}, it was Asao \cite{Asao2023} who first presented the relationship with path homology and provided a realisation of the path homology chains as the diagonal magnitude homology. 

Investigations into extensions of the path homology construction,  however, are far broader, leading to natural questions regarding cofibration category structures \cite{Carranza2024, Hepworth2024}, relations to simplicial homotopy theory \cite{Ivanov2024}, and generalised hypergraph homology \cite{Grbic2022} for use in data analysis.

More concretely, digraphs occur ubiquitously in the form of networks across many scientific disciplines, being fundamental descriptors in the modeling of complex system interactions.
In particular, in conjunction with persistent homology \cite{Carlsson09, Edelsbrunner2014}, path homology has begun to be used across a range of applications.
For instance, in the analysis of temporal networks \cite{Chowdhury2022}, demonstrating how path homology representatives reveal important complex sub-networks,
and deep learning \cite{Chowdhury2019},
where the structurally simpler directed clique complex cannot as easily extract important global features from the network parameters.
\footnote{We note that in some applications, a simplified definition of path homology is used that coincides with the original construction when the digraph does not contain any double edges.}

For the purposes of applications there has been much interest in efficiently computing path homology,
with persistent path homology being known to be stable with respect to perturbations in edge weighted filtrations of the digraph \cite{Chowdhury2018}. 
Yet, effective computation of path homology remains limited to low dimensions, with one of the most effective algorithms devised \cite{Dey2022} being only for the computation of path homology in dimension $1$ and depending primarily on the knowledge of low dimensional bases of the path chain complex.

More generally, a simple procedure for the computation of path homology is provided by Grigor'yan \cite[\S 1.7]{Grigoryan2022}, as a special case of persistent path homology
\cite[\S 5]{Chowdhury2018}.
However, the practicality of these algorithms is limited by the computation of a chain level base as it involves the computation of the null space of certain large matrices whose sizes grow rapidly with the number of digraph edges.

Path homology algorithms would be greatly simplified in the absence of the need to first compute directly a basis of the paths chains, as pointed out in \cite{Grigoryan2022} Problem 1.7.
In particular, improved computational speed of path homology in higher dimensions would greatly enhance the practicality of a wide range of applications.

The central problem lies in the fact that unlike simplicial homology, there is no simple general description of a path chain basis inherent in its construction.
More precisely,
chains are provided indirectly through compatibility with the differential, having no predetermined free module structure generated directly by singular mappings.
It is precisely such maps that are relied upon for more easily computable homologies of digraphs such as the directed clique complex \cite{Masulli2017, Reimann2017},
and digraph cubical homology \cite{Grigor'yan2021}, which like path homology is a homotopy invariant functor on digraphs satisfying a version of the Mayer–Vietoris exact sequence.
However, \cite{Grigor'yan2023} demonstrates that for a certain class of cubical digraphs embeddable in a directed (hyper)cube, the path chains and cubical chains coincide.

Denoting by $\Omega_n(G;R)$ the path chains in dimension $n$ with coefficients in commutative ring $R$,
a canonical basis of $\Omega_0(G;R)$ and $\Omega_1(G;R)$ is generated by all vertices and edges respectively.
The situation in dimension $2$ is made more complicated by the existence of multisquare digraphs with no canonical basis.
However, a basis can be chosen from a subset of generators corresponding to certain digraphs of the form of squares, triangles and double edges.

Restricting  to coefficients in a field,
Grigor'yan \cite{Grigoryan2022} gave a description of a basis of $\Omega_3(G;R)$ in terms of the images of the top dimensional generator of a trapezohedron under the constraints that the digraph contains no multisquares and double edges.
More generally when the digraph does not contain any multisquares, by dualising a formula for path cochains, Fu and Ivanov \cite{Fu2024} provided a basis of $\Omega_*(G;R)$ making use of a correspondence to certain graphical constructions.
In particular, using their method to obtain a digraph whose path homology Euler characteristic differs with coefficients in $\mathbb{Q}$ and $\mathbb{Z}_2$.
Which cannot be the case for the singular homology of a space.

In the present work, we make no assumptions on the structure of the digraph $G$.
Our primary construction is of structures we call face multigraphs and more generally face multihypergraphs, formed in dimension $n+1$ from elements of the path chains in dimensions $n$ and $n-1$.
Face multigraphs being sufficient when the coefficient ring has no additive torsion other than $2$.

Face multihypergraphs are principally considered together with their complete extensions by a vertex. The existence of which we show is sufficient to produce elements of $\Omega_{n+1}(G;R)$ from elements of the previous two dimensions.

Restricting further to strongly connected extensions, we establish the notion of inductive element of $\Omega_*(G;R)$, built inductively from the vertex basis of $\Omega_0(G;R)$ as strongly connected complete extensions over face multihypergraphs.
Inductive elements provide the structure for the following two central results concerning generating set of $\Omega_*(G;R)$.

\begin{theorem}[Corollary~\ref{cor:FieldGeneratorBasis}]
    The $n$-dimensional inductive elements generate $\Omega_n(G;\mathbb{Z}_p)$.
\end{theorem}
\begin{theorem}[Corollary~\ref{cor:IductiveBasisIntegral}]
    Inductive elements generate $\Omega_3(G;\mathbb{Z})$ and $\Omega_3(G;K)$ when $K$ is a field of characteristic $0$.
\end{theorem}

In dimensions $0$, $1$ and $2$, for the case of extensions over face multigraphs, the inductive elements we construct coincide with the previously presented natural generators up to sign. Of course, when working over a field, any generating set can be reduced to a basis. Thus we make the following statement, which at least partially addresses Grigor'yan's Problem 2.11 from \cite{Grigoryan2022} regarding the construction of bases of $\Omega_n(G;R)$ for an arbitrary $n$ under no restrictions on the digraph.
\begin{corollary}
For any $n\geq0$ and any prime $p$, the module $\Omega_n(G;\mathbb{Z}_p)$ admits a basis of inductive elements. For any characteristic $0$ field $K$, the module $\Omega_3(G;K)$ admits a basis of inductive elements.
\end{corollary}

Finally, the necessity of the conditions required in the construction of strongly connected complete extensions over face multihypergraphs is demonstrated by their use in the construction of two important examples. 
The first of these provides a sequence of digraphs whose path homology differential with respect to an inductive basis contains entries of arbitrary multiplicity. The second, contributes the following theorem.

\begin{theorem}[Example~\ref{exam:DiffEulerOverFiniteFields}]
    For any positive integer $t\geq 2$ and fields $K$ and $K'$, there is a digraph $\mathbb{E}_t$ such that 
    \[
        \dim(\Omega_{4}(\mathbb{E}_{t};K)) =
        \begin{cases}
            1 & \text{if} \; K = \mathbb{Z}_{p_j} \: \text{for} \: j \in \{ 1,\dots,k\},
            \\
            0 & \text{otherwise} 
        \end{cases}
    \]
    where $t = p_1^{i_1}\cdots p_{k}^{i_k}$ is the prime decomposition of $t$ and
    \[
        \dim(\Omega_i(\mathbb{E}_{t};K)) = \dim(\Omega_i(\mathbb{E}_{t};K')),
    \]
   for $i \neq 4$.
    
\end{theorem}

This resolves the second open problem of Fu and Ivanov contained in the introduction of \cite{Fu2024} regarding the existence of digraphs whose path Euler characteristic is different with $\mathbb{Q}$ and $\mathbb{Z}_p$ coefficients for $p \geq 3$.

\section{Background}

We first provide the necessary background material required for the rest of the paper.
Throughout this work, $\mathbb{Z}$ is the ring of integers, $\mathbb{Z}_p$ for a prime number $p$ is the finite field with $p$ elements, and $R$ denotes an arbitrary commutative ring with a unit.

\subsection{Algebraic preliminaries}

A \emph{basis} of a free module $F$ over $R$ is a family $\{f_i\mid i\in I\}$ of distinct elements $f_i \in F$, where $i$ runs over the members of some indexing set $I$, such that every element $f \in F$ can be uniquely expressed as a \emph{linear combination}
\[
    f = \sum_{i\in I} \alpha_i f_i
\]
where each \emph{coefficient} $\alpha_i \in R$ and at most finitely many $\alpha_i \neq 0$. 
In particular, every free $R$-module has a basis.

Let $V$ be a vector space over a field $K$, with basis $B=\{v_i\}_{i\in I}$ for some set $I$.
The \emph{lattice} of $B$ in $V$ consists of all points in $V$ that can be obtained as a linear combination of elements of $B$ with integer coefficients.

\subsection{Directed graphs, labeled hypergraphs, and labeled multihypergraphs}

A \emph{digraph} $G$ is a pair $(V_G,E_G)$, whose \emph{vertices} are elements of the set $V_G$ and whose \emph{edges} are element of the set
\[
    E_G \subseteq \{ (u,v) \in V_G \times V_G \: | \: u \neq v \}.
\]
Throughout this work we assume that $G$ is an arbitrary digraph, unless stated otherwise.
We also denote edges $(u,v) \in E_G$ by $u \to v$.
A digraph $G$ is called \emph{finite} if the set $V_G$ is finite. A \emph{subdigraph} of a digraph $G$ is a digraph $H$ such that
\[
    V_H \subseteq V_G 
    \;\;\; \text{and} \;\;\;
    E_H \subseteq E_G.
\]

A \emph{multihypergraph} $M$ consists of a pair $(V_M,E_M)$, whose \emph{vertices} are elements of the set $V_M$ and whose \emph{hyperedges} are elements of the multiset $E_M$, all of whose elements consist of subsets of $V_M$.
We refer to elements of $E_M$ with multiplicity $2$ as \emph{edges} of $M$.
A multihypergraph $M$ is called a \emph{multigraph} if all elements of $E_M$ have multiplicity $2$.
A multigraph $M$ is called a \emph{graph} if each member of $E_M$ has multiplicity $1$.

A \emph{vertex labeled multihypergraph} with label set $L_V$ is a multigraph $M$ together with a function $l_V \colon V_M \to L_V$.
Similarly, an \emph{edge labeled multihypergraph} with label set $L_E$ is a multihypergraph $M$ together with a function $l_E\colon E_M \to L_E$.
A multihypergraph that is both vertex labeled and edge labeled is called a \emph{labeled multihypergraph}.

A multihypergraph is said to be \emph{connected} if for any distinct vertices $u,w \in V_M$, there exists a positive integer $t$ and a sequence of hyperedges
\[
    u\in\{v_1^1,\dots, v_{m_1}^1\},\{v_1^2,\dots v_{m_2}^2\},\dots,\{v_1^{t-1},\dots v_{m_{t-1}}^{t-1}\},v\in\{v_1^t,\dots,v_{m_t}^t\} \in E_M
\]
such that $\{v_1^{i},\dots, v_{m_1}^{i}\} \cap \{v_1^{i+1},\dots, v_{m_1}^{i+1}\} \neq \emptyset$ for $i=1,\dots,t-1$.
The same terminology applies to labeled multihypergraphs, so that one such is connected in case the underlying multihypergraph is.

\subsection{Path homology}

We detail here the construction of path homology as laid out in the foundational paper \cite{Grigoryan2013}.
In the next subsection we provide an alternative characterisation of the path chain modules in terms of diagonal magnitude homology.
Throughout this section, let $G$ be a digraph and $n$ a non-negative integer, unless stated otherwise.

An \emph{elementary $n$-path} ($n\geq0$) in a set $V$ is a sequence $v_0,\dots,v_n \in V$, which we denote $e_{v_0,\dots,v_n}$.
Define $\Lambda_n(V;R)$ to be the free $R$-module generated by all elementary $n$-paths in $V$.
In addition, set $\Lambda_{n}(V;R) = 0$ for $n<0$.
Then every $x\in \Lambda_n(V;R)$ has a unique expression
\begin{equation}\label{eq:UniqueChainFormInitial}
    x = \sum_{v_0,\dots,v_n \in V} \alpha_{v_0,\dots,v_n} e_{v_0,\dots,v_n}
\end{equation}
where each $\alpha_{v_0,\dots,v_n} \in R$ and at most finitely many $\alpha_{v_0,\dots,v_n}\neq 0$.

Define maps $\partial_{n,i}^P\colon \Lambda_n(V;R) \to \Lambda_{n-1}(V;R)$ for each $i=0,\dots,n$ by linearly extending
\begin{equation*}
    \partial_{n,i}^P(e_{v_0,\dots,v_n}) = e_{v_0,\dots,\hat{v}_i,\dots,v_n}
\end{equation*}
where $v_0,\dots,\hat{v}_i,\dots,v_n$ denotes the sequence $v_0,\dots,v_n$ with the element $v_i$ removed.
The graded module $\Lambda_*(V;R)$ becomes a chain complex $(\Lambda_*(V;R),\partial^P_*)$ with differential
\[
    \partial^P_n = \sum_{i=0}^n (-1)^i\partial^P_{n,i}.
\]
We call $\partial^P_n$ the \emph{path differential}.
Clearly, the chain complex $(\Lambda_*(V;R),\partial^P_*)$ has trivial homology in all but degree $0$.

An elementary $n$-path $e_{v_0,\dots,v_n}\in \Lambda_n(V;R)$ is called \emph{regular} if $v_{i-1} \neq v_i$ for every $i=1,\dots,n$, and \emph{irregular} otherwise. Denote by $\mathcal{I}_n(V;R)$ the free $R$-module generated by the set of elementary irregular $n$-paths and define a graded module $\mathcal{R}_*(V;R)$ by
\[
    \mathcal{R}_n(V;R) = \Lambda_n(V;R) / \mathcal{I}_n(V;R).
\]
The path differential $\partial^P_*$ descends to the quotient and $(\mathcal{R}_*(V;R), \partial^P_*)$ becomes a chain complex. Still, $(\mathcal{R}_*(V;R), \partial^P_*)$ like $(\Lambda_*(V;R), \partial^P_*)$ has trivial homology in all but degree $0$.

An \emph{(allowed) $n$-path} in a digraph $G=(V_G,E_G)$ is an elementary $n$-path $e_{v_0,\dots,v_n}$ in $V_G$ such that
\[
    (v_{i-1},v_i) \in E_G
\]
for each $i = 1, \dots, n$. 
The allowed paths span a submodule of $\Lambda_*(V_G;R)$ which is mapped injectively into $\mathcal{R}_*(V_G;R)$ by the quotient projection. Let $\mathcal{A}_*(G;R)\subseteq \mathcal{R}_*(V_G;R)$ be the image of this submodule. Following a standard abuse, we call also the members of $\mathcal{A}_*(G;R)$ \emph{allowed paths}, and will denote the cosets in $\mathcal{A}_*(G;R)$ by their unique allowed representatives. A fact that we use repeatedly in the sequel is that $\mathcal{A}_*(G;R)$ is a free $R$-module.

On the other hand, the graded module $\mathcal{A}_*(G;R)$ is not generally a subcomplex of $(\mathcal{R}_*(V_G;R),\partial^P_*)$, as it need not be the case that $\partial^P_n(A_n(G;R))\subseteq A_{n-1}(G;R)$. 
Therefore, we pass to the submodule $\Omega_*(G;R)$ of $\mathcal{A}_{n}(G;R)$ defined to be
\[
    \Omega_n(G;R) = \{ x \in \mathcal{A}_n(G;R) \: | \: \partial^P_{n}(x) \in \mathcal{A}_{n-1}(G;R) \}.
\]
By construction, $\Omega_*(G;R)$ is the smallest submodule of $\mathcal{A}_{n}(G;R)$ on which $\partial^P_n$ is a differential. Following the usual convention we assume that $\Omega_n(G;R)=0$ for $n < 0$.

\begin{definition}
    We call the chain complex $(\Omega_*(G;R), \partial^P_*)$ the \emph{path chain complex} of the digraph $G$ with coefficients in $R$.
    An element of $\Omega_n(G;R)$ is called a \emph{path chain} of dimension $n$.
    The homology $H^P_*(G;R)$ of $(\Omega_*(G;R), \partial^P_*)$ is called the \emph{path homology} of the digraph $G$ with coefficients in $R$.
\end{definition}

We note that similar to the expression in equation~\eqref{eq:UniqueChainFormInitial}, it remains the case that each $x\in \Omega_n(G;R)$ can be written uniquely in the form
\begin{equation}\label{eq:UniqueChainForm}
    x = \sum_{e_{v_0,\dots,v_n} \in P^G_n} \alpha_{v_0,\dots,v_n} e_{v_0,\dots,v_n}
\end{equation}
where $P^G_n$ is the set of all $n$-paths of $G$ and $\alpha_{v_0,\dots,v_n} \in R$ with at most finitely may $\alpha_{v_0,\dots,v_n}\neq 0$.

\subsection{Magnitude homology and path homology}\label{sec:MagnitudeHomology}

Throughout this section, $G$ is a digraph, $n$ a non-negative integer, and $l$ a non-negative real number, unless stated otherwise.

Magnitude homology was introduced by Hepworth and Willerton \cite{Hepworth2017} as a homology theory for graphs.
The original definition has since been broadly generalised to quasi-metric spaces, as presented here, and further to enriched categories \cite{Leinster2021}.
The relationship between the magnitude and path homologies of a digraph was first presented by Asao \cite{Asao2023},
making use of a spectral sequence identified by Hepworth and Willerton. The first page of the spectral sequence coincides with the magnitude homology and the second page contains the path homology along a row.
Asao showed, moreover, that all pages of this spectral sequence past the first are homotopy invariants of the digraph. The modules on the second page constitute what is now known as bigraded path homology \cite{Hepworth2024}. Here we focus on the path homology chains described in the previous section and details required later in this work.

A  \emph {(extended) quasi-metric space} $(X,d)$ is a set $X$ together with a function
\[
    d \colon X \times X \to [0,\infty]
\]
such that
\begin{enumerate}[(1)]
\item $d(x_1,x_1)=0$,
\item $d(x_1,x_3)\leq d(x_1,x_2)+d(x_2,x_3)$ and
\item $d(x_1,x_2)=d(x_2,x_1)=0$ implies that $x_1=x_2$
\end{enumerate}
for all $x_1,x_2,x_3 \in X$.

Thus let $(X,d)$ be a quasi-metric space. 
We denote by $\langle x_0,\dots,x_n\rangle$ an $(n+1)$-tuple $(x_0,\dots,x_n) \in X^{n+1}$ for which $x_{i-1}\neq x_i$ for all $i=1,\dots,n$.
We write
\begin{equation}\label{eq:length}
    \ell\langle x_0,\dots,x_n\rangle=\sum^{n}_{i=1}d(x_{i-1},x_i),
\end{equation}
and call this quantity the \emph{length} of $\langle x_0,\dots,x_n\rangle$.
Define free $R$-modules $\text{C}^M_{n,l}(X;R)$ by
\begin{equation*}
    C^M_{n,l}(X;R) =
    R
    \left[
    \left\{
    \langle x_0, \dots, x_n \rangle
    \mid
    \ell\langle x_0,\dots,x_n\rangle = l
    \right\}
    \right]
\end{equation*}
for $n\geq0$, with $C^M_{n,l}(X;R) = 0$ for $n<0$.
For $i=1,\dots,n-1$ define $\partial_{i,n}^M\colon C^M_{n,l}(X;R) \to C^M_{n-1,l}(X;R)$ by linearly extending
\begin{equation*}
    \partial_{n,l,i}^M\langle x_0,\dots,x_n\rangle =
    \begin{cases} 
        \langle x_0,\dots,x_{i-1},x_{i+1},\dots, x_n \rangle
        &
        \text{if}\;\; d(x_{i-1},x_i) + d (x_i,x_{i+1}) = d(x_{i-1}, x_{i+1}), \\
        0
        &
        \text{otherwise}
    \end{cases}
\end{equation*}
and set
\begin{equation*}
    \partial_{n,l}^M = \sum_{i=1}^{n-1} (-1)^i \partial_{n,l,i}^M
\end{equation*}
making $(C^M_{*,l}(X;R),\partial^M_{*,l})$ a chain complex for each $l \in \mathbb{R}$.

\begin{definition}
    For $l\in \mathbb{R}$ define the \emph{magnitude homology} $H^M_{*,l}(X;R)$ to be the homology of the chain complex $(C^M_{*,l}(X;R), \partial_{*,l}^M)$.
\end{definition}

A digraph $G$ comes furnished with a natural quasi-metric $d_G$ given by
\begin{equation}\label{eq:GraphMetric}
d_G(u,v)=\min\{n\geq0\mid \exists \: e_{u=v_0,\dots,v_n=v}\in P^G_n\}
\end{equation}
with the understanding that $d_G(u,v)=\infty$ if there is no allowed path in $G$ from $u$ to $v$. We will generally suppress this quasi-metric from notation, writing $C^M_{n,l}(G;R)$ and $H^M_{n,l}(G;R)$ for the magnitude chains and magnitude homology of the quasi-metric space $(G,d_G)$. 
Note that in this case we need only consider integer values of $l$ to obtain all information about the magnitude homology.

The magnitude homology $H^M_{n,n}(G;R)$ for $n\geq 0$ is called the \emph{diagonal magnitude homology} of the digraph $G$. As $C^M_{n,l}(G;R)=0$ for $l<n$, it holds that
\begin{equation}\label{eq:DiagonalKernal}
    H^M_{n,n}(G;R) =\ker(\partial^M_{n,n}).
\end{equation}

\begin{lemma}[\cite{Asao2023} Lemma 6.8]\label{lem:PathChianKernel}
    The map $\phi_n \colon \mathcal{A}_n(G;R)\to C^M_{n,n}(G;R)$ given by linearly extending
    \[
        \phi_n(e_{v_0,\dots,v_n}) = \langle v_0,\dots,v_n \rangle
    \]
    is an isomorphism of $R$-modules and induces an 
    isomorphism
    \[
        \phi_n \colon \Omega_n(G;R)\to H^M_{n,n}(G;R).
    \]
\end{lemma}

\begin{proof}
    When $n=l$, condition~\eqref{eq:length} together with the structure of $G$ as a quasi-metric space implies that for any $\langle v_0,\dots,v_n \rangle \in C^M_{n,n}(G;R)$ we have $d(v_{i-1},v_i)=1$ for $i=1,\dots,n$, so $(v_{i-1},v_{i}) \in E_G$.
    Therefore,
    $\phi_n \colon \mathcal{A}_n(G;R)\to C^M_{n,n}(G;R)$
    is a well defined isomorphism, as $\mathcal{A}_n(G;R)$ and $ C^M_{n,n}(G;R)$ are free $R$-modules on generators indexed by precisely the same sequences of vertices.
    
    Now, for any allowed path $e_{v_0,\dots,v_n}$ we have that
    \begin{align*}
        & \partial^P_{n}(e_{v_0,\dots,v_n})
        =
        e_{v_1,\dots,v_n} + (-1)^n e_{v_0,\dots,v_{n-1}} +
        \sum_{i=1}^{n-1} (-1)^ie_{v_0,\dots,\hat{v}_i,\dots,v_n}
        \\ &=
        e_{v_1,\dots,v_n} + (-1)^n e_{v_0,\dots,v_{n-1}} +
        \sum_{\substack{i=1,\dots,n-1 \\ (v_{i-1},v_{i+1}) \in E_G}}
        (-1)^i e_{v_0,\dots,\hat{v}_i,\dots,v_n}
        +
        \sum_{\substack{i=1,\dots,n-1 \\ (v_{i-1},v_{i+1}) \notin E_G}}
        (-1)^i e_{v_0,\dots,\hat{v}_i,\dots,v_n}
    \end{align*}
    where the last summation contains all non-allowable elementary paths which appear in the expression.
    Moreover, these are the only elementary $(n-1)$-paths whose images under $\phi_{n-1}$ can correspond to non-trivial summands $\langle v_0,\dots,\hat{v}_i,\dots,v_n \rangle$ in the image of $\partial_{n,n,j}^M(\phi_n(e_{v_1,\dots,v_n}))$ for some $j=1,\dots,n-1$. 
    By linearly extending the above calculation, we obtain that any $x\in \mathcal{A}_n(G;R)$ satisfies $\partial_n^P(x) \in \Omega_{n-1}(G;R)$ only when $\partial_{n,n}^M(x) = 0$.
    \qed
\end{proof}

In the remainder of this work, making use of the isomorphism~\eqref{eq:DiagonalKernal}, we treat path homology chains $\Omega_n(G;R)$ as the kernel of $\partial^M_{n,n}$ under the identification 
provided by the map $\phi_n$ from Lemma~\ref{lem:PathChianKernel}.

\subsection{Basis constructions for \texorpdfstring{$\Omega_n(G;R)$}{path chains}}\label{sec:LowDimBasis}

Most of the material covered in this section is contained in \cite{Grigoryan2022}.
However, some of the content is from, or was originally presented in, other works that we cite at the corresponding parts of the section.
Throughout this section, let $n$ be a non-negative integer, $G$ a digraph, and $R$ a commutative ring with a unit.

For any vertex $v\in V_G$, $e_v$ is an allowed path and $\partial^P_0e_v = 0$.
Hence,
\[
    \{ e_v \: | \: v \in V_G \} \; \text{is a basis of} \:\: \Omega_0(G;R).
\]
Similarly, for any $(u,v)\in E_G$, the element $e_{u,v}$ is an allowed path and $\partial^P_1(e_{u,v}) = e_u - e_v \in \mathcal{A}_0(G;R)$.
Therefore, 
\[
    \{ e_{u,v} \: | \: (u,v)\in E_G  \} \; \text{is a basis of} \:\: \Omega_1(G;R).
\]

The first non-straightforward case occurs when $n=2$.
Let $v_0,v_1,v'_1,v_2 \in E_G$.
When $(v_0,v_1) \in E_G$ and $(v_1,v_0) \in E_G$ we call $e_{v_0,v_1,v_0}$ a \emph{double edge}.
We call $e_{v_0,v_1,v_2}$ a \emph{directed triangle} if $(v_0,v_1),(v_1,v_2),(v_0,v_2) \in E_G$.
Finally, we call
$e_{v_0,v_1,v_2}-e_{v_0,v'_1,v_2}$ a \emph{directed square} if $(v_0,v_1),(v_0,v'_1),(v_1,v_2),(v'_1,v_2) \in E_G$, $v_0 \neq v_2$, $v_1 \neq v'_1$, and $(v_0,v_2) \notin E_G$.
It is straightforward to check that double edges, directed squares, and directed triangles are elements of $\Omega_2(G;R)$.
Analogously, a digraph $G$ is said to contain a \emph{double edge}, a \emph{directed triangle}, or a \emph{directed square} (at the same associated vertices as above)
if it contains the respective subdigraphs
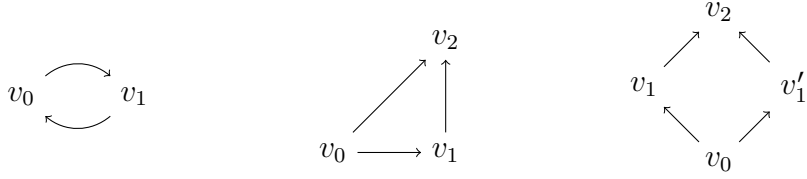
\begin{center}
    \tikz {
        \node (b) at (0,0) {$\:$};
        \node (u) at (0,1) {$v_0$};
        \node (v) at (1.5,1) {$v_1$};
        \draw[->] (u) to [out=40,in=140] (v);
        \draw[->] (v) to [out=220,in=320] (u);
        }
        \;\;\;\;\;\;\;\;\;\;\;\;\;\;\;
        \tikz {
        \node (b) at (0,0) {$\:$};
        \node (v0) at (0,0.25) {$v_0$};
        \node (v1) at (1.5,0.25) {$v_1$};
        \node (v2) at (1.5,1.75) {$v_2$};
        \draw[->] (v0) -- (v1);
        \draw[->] (v1) -- (v2);
        \draw[->] (v0) -- (v2);
        }
        \;\;\;\;\;\;\;\;\;\;\;\;\;\;\;
        \tikz {
        \node (v0) at (0,0) {$v_0$};
        \node (v1) at (-1,1) {$v_1$};
        \node (v-1) at (1,1) {$v'_1$};
        \node (v2) at (0,2) {$v_2$};
        \draw[->] (v0) -- (v1);
        \draw[->] (v0) -- (v-1);
        \draw[->] (v1) -- (v2);
        \draw[->] (v-1) -- (v2);
        }
\end{center}
and, in the last case, provided also that there is no edge $v_0 \to v_2$.
Moreover, a digraph $G$ is said to contain \emph{no double edge}, \emph{no directed triangle}, or \emph{no directed square} if it does not contain any subdigraphs of the above respective forms.

\begin{definition}
    Let $v_0,v_2\in V$ such that $v_0 \neq v_2$, $(v_0,v_2)\notin G$ and,
    \[
        S_{v_0,v_2} = \{ v_1 \in V \: | \: (v_0,v_1), (v_1,v_2) \in E_G \}.
    \]
    When $|S_{v_0,v_2}| \geq 3$, we say that $G$ contains a \emph{multisquare} between $v_0$ and $v_2$.
    If $G$ does not contain any multisquares between any pair of vertices, then we say $G$ \emph{contains no multisquares}.
\end{definition}

Variations of the following proposition have been proved for both field and integral coefficients in \cite[Proposition 4.2]{Grigoryan2013}, \cite[Proposition 2.9]{Grigor'yan2014b}, and \cite[Theorem 1.8]{Grigoryan2022}. We include a proof at the end of the next section.

\begin{proposition}[\cite{Grigor'yan2014b, Grigoryan2013, Grigoryan2022}]\label{prop:Dim2Base}
    The double edges, directed triangles, and directed squares generate $\Omega_2(G;R)$.
    Once a basis of directed squares within each multisquare is chosen, these directed squares, double edges, directed triangles, and a choice up to sign of directed squares not contained in a multisquare form a basis of $\Omega_2(G;R)$.
\end{proposition}
We note that the existence of multisquares implies that no canonical basis of $\Omega_2(G;R)$ exists in general.
Moreover, the following example demonstrates that this is also the case for $\Omega_n(G;R)$ when $n\geq 2$.
\begin{example}
    For $t\geq 3$ consider the digraph $G$ with $V_G= \{ v_0,v_1^1,v_1^2,v_1^3,v_2,v_3,\dots,v_{t-1},v_t \}$ and edges
    \begin{align*}
        E_G = & \{  (v_0,v^1_1), (v^1_1,v_2), (v^1_1,v_3), (v_0,v^2_1), (v^2_1,v_2), (v^2_1,v_3), (v_0,v^3_1), (v^3_1,v_2), (v^3_1,v_3) \}
        \\
        & \cup
        \{ (v_{i},v_{i+1}) \: | \: i=2,\dots,t-1 \}
        \cup
        \{ (v_{i},v_{i+2}) \: | \: i=2,\dots,t-2 \}.
    \end{align*}
    \begin{center}
        \tikz {
            \node (v0) at (0,0) {$v_0$};
            \node (v1) at (1,1) {$v^1_1$};
            \node (v12) at (1,-1) {$v^2_1$};
            \node (v13) at (1,2) {$v^3_1$};
            \node (v2) at (2,0) {$v_2$};
            \node (v3) at (3.414,0) {$v_3$};
            \node (v4) at (4.414,1) {$v_4$};
            \node (v5) at (5.414,0) {$v_5$};
            \node (v6) at (6.414,1) {$v_6$};
            \node (vt-2) at (9.414,0) {$v_{t-2}$};
            \node (vt-1) at (10.414,1) {$v_{t-1}$};
            \node (vt) at (11.414,0) {$v_t$};
            \node (d) at (8.414,0.9) {$\cdots$};
            \node (d) at (7.414,0.1) {$\cdots$};
            \node (B6) at (7.414,1) {};
            \node (B5) at (6.414,0) {};
            \node (B-6) at (7.014,0.4) {};
            \node (Bt-1) at (9.414,1) {};
            \node (Bt-2) at (8.414,0) {};
            \node (B-t-2) at (8.814,0.6) {};
            \draw[->] (v0) -- (v1);
            \draw[->] (v0) -- (v12);
            \draw[->] (v0) -- (v13);
            \draw[->] (v1) -- (v2);
            \draw[->] (v12) -- (v2);
            \draw[->] (v13) -- (v2);
            \draw[->] (v2) -- (v3);
            \draw[->] (v3) -- (v4);
            \draw[->] (v4) -- (v5);
            \draw[->] (v5) -- (v6);
            \draw[->] (vt-2) -- (vt-1);
            \draw[->] (vt-1) -- (vt);
            \draw[->,gray] (v1) -- (v3);
            \draw[->,gray] (v12) -- (v3);
            \draw[->,gray] (v13) -- (v3);
            \draw[->,gray] (v2) -- (v4);
            \draw[->,gray] (v3) -- (v5);
            \draw[->,gray] (v4) -- (v6);
            \draw[->,gray] (vt-2) -- (vt);
            \draw[-,gray] (v6) -- (B6);
            \draw[-] (v6) -- (B-6);
            \draw[-,gray] (v5) -- (B5);
            \draw[->,gray] (Bt-1) -- (vt-1);
            \draw[->] (B-t-2) -- (vt-2);
            \draw[->,gray] (Bt-2) -- (vt-2);
        }
    \end{center}
    The digraph $G$ contains a multisquare between $v_0$ and $v_2$.
    Each of the $t$-chains
    \[
        e_{v_0,v_1^j,v_2,v_3\dots,v_{t-1},v_t}
        -
        e_{v_0,v_1^{j'},v_2,v_3\dots,v_{t-1},v_t}
    \]
    for $j,j'=1,2,3$, and $j\neq j'$, generate $\Omega_t(G;R)$ and any two of them form a basis of $\Omega_t(G;R)$.
\end{example}

To end the section, we summarise the existing results under restrictions for bases of $\Omega_n(G;R)$ when $n\geq 3$. 
\begin{definition}\label{def:Trapezohedron}
    For $t\geq2$ let $\mathbb{T}_{t}$ be the following digraph on vertices
    \[
        V_{\mathbb{T}_{t}} =
        \{ T,u_1,\dots,u_t,v_1,\dots,v_t,H \}
    \]
    and edges
    \[
        T \to u_i,\;\;\;
        u_i \to v_i, \;\;\;
        u_i \to v_{i+1},
        \;\;\; \text{and} \;\;\;
        v_i \to H
    \]
    for $i=1,\dots,t$ where all index values are assumed to be modulo $t$.
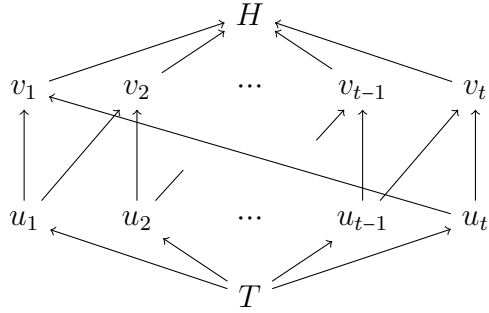
    \begin{center}
        \tikz {
            \node (T) at (2.25,0) {$T$};
            \node (u1) at (-0.75,1) {$u_1$};
            \node (u2) at (0.75,1) {$u_2$};
            \node (v1) at (-0.75,2.75) {$v_1$};
            \node (v2) at (0.75,2.75) {$v_2$};
            \node (du) at (2.25,1) {$\cdots$};
            \node (dv) at (2.25,2.75) {$\cdots$};
            \node (ut-1) at (3.75,1) {$u_{t-1}$};
            \node (vt-1) at (3.75,2.75) {$v_{t-1}$};
            \node (ut) at (5.25,1) {$u_t$};
            \node (vt) at (5.25,2.75) {$v_t$};
            \node (H) at (2.25,3.75) {$H$};
            \node (u2-) at (1.5,1.825) {};
            \node (vt-1-) at (3,1.925) {};
            \draw[->] (T) -- (u1);
            \draw[->] (T) -- (u2);
            \draw[->] (T) -- (ut-1);
            \draw[->] (T) -- (ut);
            \draw[->] (u1) -- (v1);
            \draw[->] (u1) -- (v2);
            \draw[->] (u2) -- (v2);
            \draw[->] (ut-1) -- (vt-1);
            \draw[->] (ut-1) -- (vt);
            \draw[->] (ut) -- (vt);
            \draw[->] (ut) -- (v1);
            \draw[->] (v1) -- (H);
            \draw[->] (v2) -- (H);
            \draw[->] (vt-1) -- (H);
            \draw[->] (vt) -- (H);
            \draw[-] (u2) -- (u2-);
            \draw[->] (vt-1-) -- (vt-1);
        }
    \end{center}
    The digraph $\mathbb{T}_{t}$ is called a \emph{trapezohedron} of order $t$.
\end{definition}

\begin{proposition}[\cite{Grigoryan2022} Proposition 2.1]\label{prop:Trapezohedron}
    The free $R$-module $\Omega_3(\mathbb{T}_t;R)$ has rank $1$ and $H_n^P(\mathbb{T}_t;R)=0$ for $n\geq 1$.
\end{proposition}
The unique up to sign generator of $\Omega_3(\mathbb{T}_t;R)$ is called a \emph{trapezohedron element} and can be explicitly realised as
\[
    \sum_{i=1}^t e_{T,u_i,v_i,H} - e_{T,u_i,v_{i+1},H}
\]
where $v_{t+1}=v_1$.

A \emph{map} of digraphs $f \colon H \to G$ is a function $f\colon V_H \to V_G$ such that for all $(u,v) \in E_H$ either $f(u) = f(v)$ or $(f(u),f(v)) \in E_G$.

\begin{theorem}[\cite{Grigoryan2022} Theorem 2.10]\label{thm:Dim3BasisNoDoubleNoMulti}
    Let $K$ be a field and $G$ a digraph containing no double edges and no multisquares.
    Then there is a basis of $\Omega_3(G;K)$ consisting of
    elements obtained as the induced image of a trapezohedron element under a digraph map from $\mathbb{T}_t \to G$ for some integer $t \geq 2$. 
\end{theorem}

Let $K$ be a field. Generalising the basis description above, under only the assumption that $G$ contains no multisquares, Fu and Ivanov~\cite[Theorem 4.7]{Fu2024} give an explicit description of a basis of $\Omega_n(G;K)$ in terms of certain connected components of labeled graphs called \emph{short move} graphs~\cite[Definition 3.1]{Fu2024}.
In particular, the short move graphs have vertices labeled with $n$-paths in $G$ and positive integer edge labels correspond to a single vertex difference between path at that position index. 
Moreover, the construction of Fu and Ivanov's basis is unique up to sign.

\section{Structure morphisms on the path chain complex}\label{sec:StructureMaps}

Throughout this section let $G$ be a digraph equipped with the quasi-metric $d_G$ given in equation~\eqref{eq:GraphMetric}, $n$ a non-negative integer, and $R$ a commutative ring with a unit, unless stated otherwise. In the section we define several morphisms on $\Omega_n(G;R)$ that are fundamental to the constructions presented in the rest of the paper.

\subsection{A further characterisation of \texorpdfstring{$\Omega_n(G;R)$}{path chains}}

We first provide the following lemma, which gives a characterisation of the elements of $\Omega_n(G;R)$, further refining that given in Lemma~\ref{lem:PathChianKernel}.
We are not aware of the lemma having been proved previously in the form presented.
However, a related statement is made in \cite[Lemma 4.1]{Grigoryan2013}.

\begin{lemma}\label{lem:NoPositionSwap}
    Let $x \in \mathcal{A}_n(G;R)$.
    Then $x\in \Omega_n(G;R)$ if and only if
    \[
        \partial^M_{n,n,i}(x) = 0 
    \]
    for each $i=1,\dots,n-1$.
\end{lemma}

\begin{proof}
    Sufficiency is clear from Lemma~\ref{lem:PathChianKernel}, as by definition $\partial_{n,n}^M = \sum_{i=1}^{n-1}(-1)^{i}\partial_{n,n,i}^M$.
    To show necessity we adopt the notation for magnitude homology from Section~\ref{sec:MagnitudeHomology}, identifying $\mathcal{A}_n(G;R)$ with $C^M_{n,n}(G;R)$ and $\Omega_n(G;R)$ with $H^M_{n,n}(G;R)$ using the map $\phi_n$ from Lemma~\ref{lem:PathChianKernel}. Thus if $x\in \Omega_n(G;R)$, then, similarly to equation~\eqref{eq:UniqueChainForm}, $x$ can be written uniquely as
    \begin{equation*}
        x = \sum_{j \in J} \alpha_j \langle v_0^j,\dots,v_n^j \rangle
    \end{equation*}
    where $J$ is a finite set, $e_{v_0^j,\dots,v_n^j}$ is an allowed $n$-path, and each $\alpha_j \in R$ is nonzero for all $j \in J$. 
    
    We claim that it is enough to show that for any $i=1,\dots,n-1$ and $j \in J$ satisfying
    \begin{equation}\label{eq:MagnitudeCondition}
        (v_{i-1}^j,v_{i+1}^j)\notin E_G
        \;\;\; \text{and} \;\;\;
        v_{i-1}^j \neq v_{i+1}^j
    \end{equation}
    there is a $K^j_i \subseteq J \setminus \{j\}$ with $(v^k_{i-1},v^k_{i+1}) \notin E_G$ and $v_{i-1}^k \neq v_{i+1}^k$ for each $k\in K^j_i$
    such that
    \[
        \sum_{k\in K^j_i \cup \{ j\}} \alpha_k
        \langle x_0^k,\dots,\hat{x}_{i}^k,\dots,x_n^k \rangle
        = 0.
    \]
    Indeed, if this holds, then it can be applied a second time to the set $J\setminus (K^i_j \cup \{j\})$. Successive applications will eventually yield the empty set.
    
    To show the claim, note that $x \in \ker \partial^M_{n,n}$ by Lemma~\ref{lem:PathChianKernel}, so for any $i=1,\dots,n-1$ and $j \in J$, there is a minimally sized $K \subseteq J \setminus \{j\}$
    and an $S_k\subseteq\{1,\dots,n-1\}$ for each $k\in K$
    such that
    \begin{equation}\label{eq:GeneralBoundarySum}
        (-1)^i\alpha_j
        \langle v_0^j,\dots,\hat{v}_i^j,\dots,v_n^j \rangle
        + \sum_{k\in K} 
        \sum_{s_k\in S_{k}}
        (-1)^{s_k} \alpha_k
        \langle v_0^k,\dots,\hat{v}_{s_k}^k,\dots,v_n^k \rangle
        = 0.
    \end{equation}
    Suppose that for some $k \in K$ and $s_{k}\in S_{k}$ it is the case that $s_{k} \neq i$.
    In particular, by the minimal size of $K$ we can assume
    \[
        \langle v_0^j,\dots,\hat{v}_i^j,\dots,v_n^j \rangle
        =
        \langle v_0^k,\dots,\hat{v}_{s_k}^k,\dots,v_n^k \rangle.
    \]
    Without loss of generality assume $i<s_{k}$. Then
    \begin{align}\label{ChainEqualityConditions}
         v_0^j,\dots,v_{i-1}^j = & \: v_0^{k},\dots,v_{i-1}^{k}
         \nonumber
        \\ \text{while} \: 
         v_{i+1}^j,\dots,x_{s_{k}}^j = & \: x_i^{k},\dots,x_{s_{k}-1}^{k}
        \\ \text{and} \:
         v_{s_{k}+1}^j,\dots,v_{n}^j = & \: v_{s_{k}+1}^{k},\dots,v_{n}^{k}.
         \nonumber
    \end{align}
    As $\ell(\langle v_0^j,\dots,v_n^j \rangle) = n$ and
    the distance on the vertices of $G$ comes from the digraph quasi-metric, we must have that $d_G(v_t^j,v_{t+1}^j) = 1$ and $d_G(v_t^{k},v_{t+1}^{k}) = 1$ for $t=0,\dots,n-1$.
    However, the construction of the digraph quasi-metric also implies that
    \begin{equation}\label{eq:xMetricCondition}
        d_G(v_{i-1}^j,v_{i+1}^j) = d_G(v_{i-1}^{k},v_i^{k}) = 1.
    \end{equation}
    by using the last position of the first line of equation~\eqref{ChainEqualityConditions} and the first position of the second line.
    Meanwhile, by the assumption in equation~\eqref{eq:MagnitudeCondition}, we also have that
    \[
        d_G(v_{i-1}^j,v_{i+1}^j) = d_G(v_{i-1}^j,v_i^j)+d_G(v_i^j,v_{i+1}^j) = 2
    \]
    which contradicts equation~\eqref{eq:xMetricCondition}.
    Therefore, $s_k = i$ for all $s_k \in S_K$ and all $k\in K$. Together with equation \eqref{eq:GeneralBoundarySum} the proof is complete by setting $K_i^j = K$.
     \qed
\end{proof}

\subsection{Bigrading of path chains and face maps at a vertex}

The material presented in this section, which is related to connectedness and bigradings of $\Omega_n(G;R)$, constitutes a generalisation of parts of~\cite[Section 2.2]{Grigoryan2022} to the case of non-field coefficients.
In the reminder of the section, given $x \in \Omega_n(G;R)$ or $x \in \mathcal{A}_n(G;R)$, we assume that $x$ has the form
\begin{equation}\label{genformx}
    x = \sum_{e_{v_0,\dots,v_n}\in P^G_v} \alpha_{v_0,\dots,v_n} e_{v_0,\dots,v_n}
\end{equation}
where each $\alpha_{v_0,\dots,v_n} \in R$ and only finitely many $\alpha_{v_0,\dots,v_n} \neq 0$. 

\begin{definition}
    Let $x \in \Omega_n(G;R)$. Define the \emph{head set} $\bar{h}(x)$ and the \emph{tail set} $\bar{t}(x)$ to be
    \begin{align*}
        \bar{h}(x) &=
        \{ v_n \: | \: e_{v_0,\dots,v_n} \in P^G_v \; \text{and} \; \alpha_{v_0,\dots,v_n} \neq 0 \}
        \\ \;\;\; \text{and} \;\;\;
        \bar{t}(x) &=
        \{ v_0 \: | \: e_{v_0,\dots,v_n} \in P^G_v \; \text{and} \; \alpha_{v_0,\dots,v_n} \neq 0 \}
        .
    \end{align*}
    If $|\bar{h}(x)| = 1$, then $x$
    is called \emph{upper connected} and the unique element $h(x) \in \bar{h}(x)$ is called the \emph{head} of $x$.
    Similarly, if $|\bar{t}(x)| = 1$, then $x$
    is called \emph{lower connected} and the unique element $t(x) \in \bar{t}(x)$ is called the \emph{tail} of $x$.
    The element $x$ is called \emph{connected} if it is both upper connected and lower connected.
\end{definition}

Recall from Lemma~\ref{lem:PathChianKernel} that there are isomorphisms $\phi_n\colon \mathcal{A}(G;R)\to C_{n,n}^M(G;R)$ and $\phi_n\colon \Omega(G;R)\to H_{n,n}^M(G;R)$. From equation~\eqref{eq:DiagonalKernal} it follows that 
\[    
\ker(\partial^M_{n,n})= H^M_{n,n}(G;R) \cong \Omega(G;R)
\]
where $\partial^M_{n,n}$ is the magnitude differential.
Since $\partial^M_{n,n}$ preserves the endpoints of the $(n+1)$-tuples in the magnitude chain complex, $\ker(\partial^M_{n,n})$ admits a basis consisting of elements which are sums of tuples having the same endpoints. These elements correspond under the isomorphism above to elements of $\Omega(G;R)$ which are sums of paths having the same endpoints. Therefore, $\Omega_n(G;R)$ admits a basis of connected elements (compare \cite[Lemma 2.2]{Grigoryan2022}, where a similar statement was made for field coefficients).

It now follows that for any vertices $v_h,v_t\in V_G$, the connected elements with head $v_h$ and tail $v_t$ span a free submodule $\Omega_n^{v_t,v_h}(G;R)$  of $\Omega_n(G;R)$. Since the connected elements with different endpoints are linearly independent, there is a decomposition
\begin{equation}\label{eq:Bigrading}
    \Omega_n(G;R) = \bigoplus_{v_t,v_h \in V_G} \Omega_n^{v_t,v_h}(G;R).
\end{equation}
which we use to define a bigrading on $\Omega_n(G;R)$.

We say that a basis of $\Omega_n(G;R)$ \emph{respects the bigrading} $\Omega_*^{*,*}(G;R)$ if it restricts to a basis on each $\Omega_n^{v_t,v_h}(G;R)$. From the discussion above, it follows that $\Omega_n(G;R)$ always admits a basis that respects the bigrading.

\begin{definition}
    For $v\in V_G$ define homomorphisms
    \[
        \delta^h_{n,v},\: \delta^t_{n,v} \colon 
        \mathcal{A}_n(G;R) \to \mathcal{A}_{n-1}(G;R)
    \]
    by
    \begin{equation*}
        \delta_{n,v}^h(x)  =
        \sum_{\substack{e_{v_0,\dots,v_n}\in P^G_n \\ v_{n-1}=v}}
        \alpha_{v_0,\dots,v_n} e_{v_0,\dots,v_{n-1}} 
    \end{equation*}
    and
    \begin{equation*}
        \delta_{n,v}^t(x)  =
        \sum_{\substack{e_{v_0,\dots,v_n}\in P^G_n \\ v_1=v}}
        \alpha_{v_0,\dots,v_n} e_{v_1,\dots,v_n}
    \end{equation*}
    where we assume that $x \in \mathcal{A}_n(G;R)$ has the form \eqref{genformx}.
\end{definition}

\begin{lemma}\label{lem:SumFace}
    For each $v\in V_G$ the functions $\delta_{n,v}^h,\:\delta_{n,v}^t\colon \mathcal{A}_n(G;R) \to \mathcal{A}_{n-1}(G;R)$ restrict to well defined $R$-module homomorphisms
    \begin{equation}\label{eq:OmegaBimodule}
        \delta_{n,v}^h,\:\delta_{n,v}^t \colon
        \Omega_n(G;R)\to \Omega_{n-1}(G;R).
    \end{equation}
\end{lemma}

\begin{proof}
    We prove only the part of the statement regarding $\delta_{n,v}^h$.
    Thus suppose that $x\in \Omega_n(G;R) \subseteq \mathcal{A}_n(G;R)$.
    Applying Lemma~\ref{lem:NoPositionSwap} gives $\partial^M_{n,n,i}(x)=0$ for each $i=1,\dots,n-1$, which, by the definition of $\delta_{n,v}^h$, implies that $\partial^M_{n,n,i}(\delta_{n,v}^h(x))=0$ for each $i=1,\dots,n-2$.
    This shows that $\delta_{n,v}^h(x) \in \Omega_{n-1}(G;R)$, again by applying Lemma~\ref{lem:NoPositionSwap}.
    \qed
\end{proof}

Observe that the homomorphisms $\delta_{n,v}^h$ and $\delta_{n,v}^t$ respect the $\Omega_*^{*,*}(G;R)$ bigrading.
That is, $\delta_{n,v}^h,\delta_{n,v}^t$ restrict to $R$-module homomorphisms
\begin{equation*}
    \delta_{n,v}^h \colon
    \Omega_n^{u,w}(G;R)\to \Omega_{n-1}^{u,v}(G;R)
    \;\;\; \text{and} \;\;\;
    \delta_{n,v}^t \colon
    \Omega_n^{u,w}(G;R)\to \Omega_{n-1}^{v,w}(G;R)
\end{equation*}
for any $u,v,w\in V_G$. Moreover, for any $x\in \Omega_n(G;R)$ and $v \in V_G$, the element $\delta_{n,v}^h(x)$ is upper connected if nonzero, and $\delta_{n,v}^t(x)$ is lower connected if nonzero. In addition, we obtain homomorphisms
\[
    \delta^h_n(x),\: \delta^t_n(x)\; \colon 
    \Omega_n(G;R) \to \Omega_{n-1}(G;R)
\]
given by 
\[
    \delta^h_n(x) = \sum_{v\in V_G} \delta_{n,v}^h(x)
    \;\;\; \text{and} \;\;\;
    \delta^t_n(x) = \sum_{v\in V_G} \delta_{n,v}^t(x)
\]
that coincide with the maps $(-1)^n\partial_{n,n}^P$ and $\partial_{n,0}^P$, respectively.
Finally,
we verify that $\delta^h_n$ and $\delta^t_n$ interact as expected with respect to change of coefficients from $\mathbb{Z}$ to a field of characteristic $0$ and prove Proposition~\ref{prop:Dim2Base} from the previous section.

\begin{lemma}\label{lem:CoeffChange}
    Let $K$ be a field of characteristic $0$. Then the map
    \[
        \mu_* \colon \Omega_*(G;\mathbb{Z}) \otimes K \to \Omega_*(G;K)
    \]
    given by $\mu_n(x \otimes k)  = kx$ is an isomorphism of graded modules such that the following diagrams commute.
    \[
        \xymatrix{
            \Omega_n(G;\mathbb{Z}) \otimes K
            \ar[r]^-{\mu_n}
            \ar[d]_-{\delta^h_n}
            &
            \Omega_n(G;K)
            \ar[d]^-{\delta^h_n}
            \\
            \Omega_{n-1}(G;\mathbb{Z}) \otimes K
            \ar[r]^-{\mu_{n-1}}
            &
            \Omega_{n-1}(G;K)
        }
        \;\;\;\;\;\;\;\;\;\;\;\;\;\;
        \xymatrix{
            \Omega_n(G;\mathbb{Z}) \otimes K
            \ar[r]^-{\mu_n}
            \ar[d]_-{\delta^t_n}
            &
            \Omega_n(G;K)
            \ar[d]^-{\delta^t_n}
            \\
            \Omega_{n-1}(G;\mathbb{Z}) \otimes K
            \ar[r]^-{\mu_{n-1}}
            &
            \Omega_{n-1}(G;K)
        }
    \]
\end{lemma}

\begin{proof}
    The map $\mu_n$ is an isomorphism of modules, as $\Omega_n(G;\mathbb{Z})$ is a free $\mathbb{Z}$-module.
    The fact that the diagrams commute follows directly from the definitions of $\mu_*$, $\delta_*^h$ and $\delta_*^t$.
    \qed
\end{proof}

Suppose $K$ is a field of characteristic $0$.
Using the isomorphism $\mu_n$ from the lemma above the composite $\Omega_n(G;\mathbb{Z})\hookrightarrow \Omega_n(G;\mathbb{Z}) \otimes K \xrightarrow{\mu_n} \Omega_n(G;K)$ can be interpreted as realising $\Omega_{n}(G;\mathbb{Z})$ as a lattice in $\Omega_{n}(G;K)$.

\begin{proof}[Proof of Proposition~\ref{prop:Dim2Base}]
    Using the free module structure in equation~\eqref{eq:Bigrading}, the result follows by obtaining generating sets and bases of each free module $\Omega_2^{v_t,v_h}(G;R)$ for some $v_t,v_h\in V_G$.
    If $d_G(v_t,v_h) = 0$, then $v_t = v_h$ and $\Omega_2^{v_t,v_h}(G;R)$ is generated by double edges and these also form a basis.
    If $d_G(v_t,v_h) = 1$, then either $\Omega_2^{v_t,v_h}(G;R) = 0$ or it is generated by a set of directed triangles which also form a basis.
    If $d_G(v_t,v_h)>2$, then $\Omega_2^{v_t,v_h}(G;R) = 0$.
    Hence, assume that $d_G(v_t,v_h) = 2$ and $\Omega_2^{v_t,v_h}(G;R) \neq 0$.
    Then either there is a multisquare between $v_t$ and $v_h$, or $\Omega_2^{v_t,v_h}(G;R)$ is generated by a directed square and the result follows similarly to the previous cases.
    \qed
\end{proof}

\section{Upper and lower extensions over face multihypergraphs}\label{sec:Extensions}

In this section we introduce the fundamental constructions on the path chain complex which are needed in the rest of the paper.
Throughout the section, let $G$ be a digraph, $n$ a non-negative integer, and $R$ a commutative ring with unit, unless stated otherwise.

\subsection{Extensions and complete face multihypergraphs}

We introduce here the processes of forming upper and lower extensions, which are methods of constructing upper and lower connected elements of $\mathcal{A}_{n+1}(G;R)$ from elements of $\Omega_n(G;R)$.
After this, we define face multigraphs and the more complex generalisation face multihypergraphs. These objects contain the data needed to ensure that the extensions formed belong to $\Omega_{n+1}(G;R)$.

The face multigraphs and face multihypergraphs serve essentially the same purpose, but it becomes essential to work with face multihypergraphs when working over a coefficient field  of odd prime characteristic. Since the face multigraphs we define are much easier to work with, and are sufficient when working with $\mathbb{Z}$, $\mathbb{Z}_2$, or $\mathbb{Q}$ coefficients, we feel it useful to develop their theory in parallel.

To end the section we provide a notion of completeness of a face multihypergraph with respect to a vertex.
The main result of the subsequent subsection states that the existence of a complete face multihypergraph is sufficient to obtain an extension within $\Omega_{n+1}(G;R)$.

\begin{definition}
    Let $x \in \Omega_{n}(G;R)$ have the form
    \begin{equation*}
        x = \sum_{e_{v_0,\dots,v_n}\in P^G_n} \alpha_{v_0,\dots,v_n} e_{v_0,\dots,v_n}
    \end{equation*}
    where each $\alpha_{v_0,\dots,v_n} \in R$.
    Define the \emph{upper extension} $[x]^v$ of $x$ by $v\in V_G$ to be
    \[
        [x]^v =
        \sum_{\substack{e_{v_0,\dots,v_n,v}\in P^G_{n+1} 
        }}
        \alpha_{v_0,\dots,v_n}
        e_{v_{0},\dots,v_{n}, v}
        \in \mathcal{A}_{n+1}(G;R).
    \]
    An \emph{upper extension} of $x$ is any upper extension by some $v \in V_G$.
    Similarly, 
    define the \emph{lower extension} $[x]_u$ of $x$ by $u \in V_G$ to be
    \[
        [x]_u =
        \sum_{\substack{e_{u,v_0,\dots,v_n}\in P^G_{n+1} 
        }}
        \alpha_{v_0,\dots,v_n}
        e_{u, v_{0},\dots,v_{n}}
        \in \mathcal{A}_{n+1}(G;R).
    \]
    A \emph{lower extension} of $x$ is any lower extension by some $u \in V_G$.
    Any upper extension or lower extension of $x$ will be referred to as an \emph{extension} of $x$.
\end{definition}
Note that the only nonzero summands of the upper extension $[x]^v$ defined above are formed from paths $e_{v_0,\dots,v_n}$ when $\alpha_{v_0,\dots,v_n} \neq 0$ and $(v_n,v)\in E_G$, where in particular $v_n\in \bar{h}(x)$. A similar statement can also be made for lower extensions and tail sets. 

In the sequel, we develop in parallel two separate notions of upper and lower face multigraphs corresponding to the use of upper and lower extensions above, which make use of the maps $\delta^h_{n,v}$ and $\delta^t_{n,v}$ for $v\in V_G$, respectively.
We focus primarily on the upper case, explaining any differences for the lower case in brackets immediately after.

\begin{definition}\label{def:FaceMultiGraph}
    Let $n\geq 0$ and $x_1,\dots,x_m\in\Omega_n(G;R)$ be such that $x_i \neq - x_j$ for any $i,j = 1,\dots,m$ and $i\neq j$.
    We define an \emph{upper (lower) face multigraph} on $x_1,\dots,x_m$ to be a labeled multigraph $F_G^n(x_1,\dots,x_m)$ together with a choice, for each $u\in V_G$ and $i = 1,\dots,m$ such that $\delta_{n,u}^h(x_i)\neq 0$ ($\delta_{n,u}^t(x_i)\neq 0$), of $x_i^{u,1},\dots,x_i^{u,m_i^u} \in \Omega_{n-1}(G;R)$ such that
    \begin{equation}\label{eq:n-1Faces}
        \delta^h_{n,u}(x_i) =
        x_i^{u,1} + \cdots + x_i^{u,m_i^u}
        \;\;\;
        (\delta^t_{n,u}(x_i) =
        x_i^{u,1} + \cdots + x_i^{u,m_i^u})
    \end{equation}
    where no sub-sequence of $x_i^{u,1},\dots,x_i^{u,m_i}$ sums to zero. This data is to be subject to the following conditions.
    \begin{enumerate}
        \item 
        The multigraph $F_G^n(x_1,\dots,x_m)$ has $m$ vertices which are labeled with a bijection to 
        \[
            \{x_1,\dots,x_m\}.
        \]
        \item
        There can be an edge between distinct vertices $x_i,x_j$ only if $\delta_{n,u}^h(x_i)\neq 0$ and $\delta_{n,u}^h(x_j)\neq 0$ ($\delta_{n,u}^t(x_i)\neq 0$ and $\delta_{n,u}^t(x_j)\neq 0$) for some $u \in V_G$. If an edge between these vertices exists, then it is labeled with a pair
        \[
            (\{x_i^{u,k},x_j^{u,l}\},u)
            \;\;\; \text{such that} \;\;\;
            x_i^{u,k} = -x_j^{u,l}
        \]
        where $k\in \{1,\dots,m_i^u$\} and $l \in \{1,\dots,m_j^u\}$.
        \item 
        For $u \in V_G$ and $i=1,\dots,m$ such that $\delta_{n,u}^h(x_i)\neq 0$ and $k=1,\dots,m_i$, each $x_i^{u,k}$ appears in no more than one edge label.
    \end{enumerate}
\end{definition}
When we draw face multigraphs, for simplicity, an edge labeled $(\{x_i^{u,k},x_j^{u,l}\},u)$ is often written only with the label $x_i^{u,k}$ or $x_i^{u,l}$ if no ambiguity can arise.
For example, when the $x_i^{u,k}$ are upper connected with head vertex $u$. Examples of face multigraphs are constructed and drawn in the next two subsections.
\begin{remark}
    Digraphs containing multisquares are the simplest examples demonstrating the necessity of allowing for the decomposition of the form in equation~\ref{eq:n-1Faces} as part Definition~\ref{def:FaceMultiGraph}.
    For example, we will later want to consider the case when the $x_i^{u,1}$ are elements of a basis.
    When $n=3$ and there is a multisquare between vertices $w$ and $u$, then a non-zero $\delta^h_{3,u}(x_i)$ lies in a $\Omega_n^{w,u}(G;R)$ for which a basis consists of at least two elements.
    A digraph where a face multigraph is required to contain multiple vertices labeled with the same element is given later in Example~\ref{exam:ArbitraryMultiplicities}.
\end{remark}

Definition~\ref{def:FaceMultiGraph} does not quite capture all that is needed when the ground ring $R$ has additive torsion of order greater than $2$. To remedy this we introduce also the notion of a \emph{face multihypergraph}. It is easy to see that the following definition extends~\ref{def:FaceMultiGraph} in that the face multihypergraphs introduced there reduce to face multigraphs when all additive torsion in $R$ has order $2$.
Since in many applications the ring $R$ is additively torsion free
or is the two-element ring $\mathbb{Z}_2$, we continue to develop also the less complex face multigraphs alongside the face multihypergraphs.

\begin{definition}\label{def:FaceMultihyperGraph}
    Let $n\geq 0$. Using the same notation and data as Definition~\eqref{def:FaceMultiGraph}, define an \emph{upper (lower) face multihypergraph} $F^n_G(x_1,\dots,x_m)$ to be any labeled multihypergraph constructed subject to the constraints (1) and (2) of Definition~\eqref{def:FaceMultiGraph} as well as the following additional conditions
    \begin{enumerate}
        \item[(3)]
        There can be a hyperedge between vertices $x_{i_1},\dots,x_{i_t}$ for $t\geq 3$ that are not all equal, only if $\delta_{n,u}^h(x_{i_j})\neq 0$ ($\delta_{n,u}^t(x_{i_j})\neq 0$) for $j=1,\dots,t$ and some $u \in V_G$.
        If a hyperedge between these vertices exists, then it is labeled with a pair
        \[
            (\{x_{i_1}^{u,k_1},\dots,x_{i_t}^{u,k_t}\},u)
            \;\;\; \text{such that} \;\;\;
            x_{i_1}^{u,k_1} = \cdots =x_{i_t}^{u,k_t}
            \;\;\; \text{and} \;\;\;
            x_{i_1}^{u,k_1} + \cdots + x_{i_t}^{u,k_t} = 0
        \]
        where $k_j\in \{1,\dots,m_{i_j}^u$\}.
        \item[(4)]
        For $u \in V_G$ and $i=1,\dots,m$ such that $\delta_{n,u}^h(x_i)\neq 0$ ($\delta_{n,u}^t(x_{i_j})\neq 0$) and $k=1,\dots,m_i$, each $x_i^{u,k}$ appears in no more than one hyperedge label using either parts (2) or (3).
    \end{enumerate}
\end{definition}

Face multihypergraphs will appear in this work only when we are simultaneously considering an extension of some form over the sum of their vertex labels, as is made precise in the following definition.
From now on we usually drop words upper and lower when it is clear which type of face multigraph or multihypergraph is under consideration.

We also set out the notions of properness and completeness with respect to an extension for a face multihypergraph. The complete extensions will be the central notion used in the reminder of the paper and it is these extensions which we later show provide a construction of path chains in one dimension higher.
Although it would be possible to package face multigraphs, proper extensions, and complete extension together as one concept, we choose to introduce these objects separately, as we believe these notions might be used independently in the construction of path homology algorithms.

\begin{definition}\label{def:extoffmhg}
    Given an upper (lower) face multihypergraph $F_G^n(x_1,\dots,x_m)$ and $v\in V_G$, 
    the upper extension $[x_1+\cdots+x_m]^v$ (lower extension $[x_1+\cdots+x_m]_v$) is called an \emph{upper (lower) extension} by $v$ over $F_G^n(x_1,\dots,x_m)$, when $(v_i,v)\in E_G$ ($(v,v_i)\in E_G$) for each $v_i \in \bar{h}(x_i)$ ($v_i \in \bar{t}(x_i)$) and $i=1,\dots,m$.
    
    Using the notation of Definitions~\ref{def:FaceMultiGraph} and ~\ref{def:FaceMultihyperGraph}, an upper (lower) extension by $v$ over $F_G^n(x_1,\dots,x_m)$ is said to be $v$\emph{-proper} if all hyperedges $(\{x_{i_1}^{u,k_1},\dots,x_{i_t}^{u,k_t}\},u)$ corresponding to $u\in V_G$ are such that 
    \[
        (u,v) \notin E_G \;\;\; ((v,u) \notin E_G) \;\;\; \text{and} \;\;\; u \neq v.
    \]
    The upper extension of $v$ over $F_G^n(x_1,\dots,x_m)$ is further called $v$-\emph{complete} if the converse also holds. That is, if for any
    $i=1,\dots,m$ and $u \in \bar{h}(\delta^h(x_i)) \setminus\{v\}$ ($u \in \bar{t}(\delta^t(x_i)) \setminus\{v\}$)
    with $(u,v)\not\in E_G$ ($(v,u)\not\in E_G)$), $F^n_G(x_1,\dots,x_n)$ has
    for each $k=1,\dots,m_i^u$ a hyperedge labeled $(\{x_{i_1=1}^{u,k_1=k},\dots,x_{i_t}^{u,k_t}\},u)$ for some $k_j=1,\dots,m_{i_j}^u$, $j=2,\dots,t$.
\end{definition}
The terminology introduced in Definition~\ref{def:extoffmhg} is equally applicable to face multigraphs and in the sequel will be used in this context freely.
\begin{remark}\label{rmk:EdgeLables}
    Suppose $x \in \Omega_n(G;R)$ and $u,v \in V_G$, with either $(u,v) \in E_G$, or $u=v$ and $\delta_{n,u}^h(x) \neq 0$.
    For any $v$-proper face multihypergraph that can be upper extended by $v$, the definition of $v$-properness implies that it contains no edges or hyperedges corresponding to $\delta^h_{n,u}(x)$.
    In this case we can ignore the decomposition of $\delta_{n,u}^h(x)$ in equation~\eqref{eq:n-1Faces} of Definition~\ref{def:FaceMultiGraph}, as all possible choices yield isomorphic face multigraphs.
\end{remark}

Later, in Section~\ref{sec:Connectedness}, Definitions~\ref{def:connectedeness}, \ref{def:connectedeness2} we provide a further, important property for face multigraphs and face multihypergraphs. This is a notion of connectedness that is stronger than the notion of connectedness of the underlying multigraph or multihypergraph itself.

\subsection{Initial examples}

In this section we give a number of simple explicit and more general examples of extensions over face multigraphs and face multihypergraphs. 

\begin{exmp}\label{exam:VerticesAndLines}
    For any $v\in V_G$, an element $x \in \Omega_1(G;R)$ which is obtained as an upper extension over a connected $v$-complete face multigraph must be extended over a multigraph which consists of a single vertex,
    since
    \[
        \delta_{n,u}^h\delta^h_{n,v}(x) = 0
    \]
    for any $u\in V_G$.
    As discussed in Section~\ref{sec:LowDimBasis}, the set $\{e_u\mid u\in V_G\}$ is a basis of $\Omega_0(G;R)$.
    The extension described above over the basis elements $e_u \in \Omega_0(G;R)$ is a face multigraph consisting of a single vertex labeled $e_u$ such that $(u,v)\in E_G$.
    That is, such extensions correspond to edges in $G$ that begin at $u$ and end at some $v\in V_G$.
    In particular, recall again from Section~\ref{sec:LowDimBasis}  that the set of $e_{u,v}$ such that $(u,v)\in E_G$ forms a basis of $\Omega_1(G;R)$.
    
    One may consider also extensions over face multihypergraphs in this context, but it is easy to see that this gives rise to no additional extensions, and the statements above remain true in this case.
\end{exmp}

\begin{exmp}\label{exam:DirectedSquare}
    Suppose that $G$ is a digraph containing the subdigraph
    
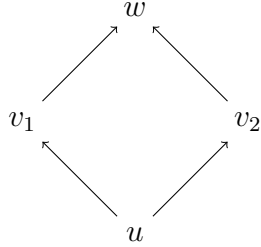
\begin{center}
        \tikz {
            \node (a) at (0,0) {$u$};
            \node (b1) at (-1.5,1.5) {$v_1$};
            \node (c) at (0,3) {$w$};
            \node (b2) at (1.5,1.5) {$v_2$};
            \draw[->] (a) -- (b1);
            \draw[->] (b1) -- (c);
            \draw[->] (b2) -- (c);
            \draw[->] (a) -- (b2);
        }
    \end{center}
    such that $(u,w) \notin G$.
    We have that $e_{u,v_1}, e_{u, v_2} \in \Omega_1(G;R)$ and $e_u \in \Omega_0(G;R)$.
    The directed square
    \[
        e_{u,v_1,w} - e_{u,v_2,w} \in \Omega_2(G;R)
    \]
    can be obtained as the upper extension $[e_{u,v_1}-e_{u, v_2}]^w$.
    Furthermore, this upper extension is an upper extension over the face multigraph
    \begin{center}
        \tikz {
            \node (b1) at (0,0) {$e_{u,v_1}$};
            \node (b2) at (3,0) {$- e_{u, v_2}$};
            \draw[-] (b1) -- (b2) node[midway,above] {$e_u$};
        }
    \end{center}
    where  we apply remark~\ref{rmk:EdgeLables} to the edge label, whose full label would be $(\{e_u, -e_u \},u)$.
    In fact, the above upper extension is $w$-complete over the face multigraph, as $(u,w) \notin G$ and $u \neq w$ with
    \[
        \delta_{n,u}^h(e_{u,v_1}) =
        \delta_{n,u}^h(e_{u,v_2}) =
        e_a.
    \]
    Moreover,
    \[
        \delta_{n,v}^h(e_{u,v_1}) =
        \delta_{n,v}^h(e_{u,v_2}) =
        0
    \]
    for any $v\in V_G$ such that $v \neq u$.
    Therefore, the face multigraph above is the unique connected, $w$-complete face multigraph containing vertices $e_{u,v_1}$ and $-e_{u,v_2}$.
    Similarly, it can be seen that directed triangle or double edge are extensions over a face multigraph consisting of a single vertex.
    Moreover, the example is again equally valid for multihypergraphs. 
\end{exmp}

\begin{exmp}\label{ex:SimpleMultiedge}
    Let $G$ be the following digraph containing a single multisquare.
    
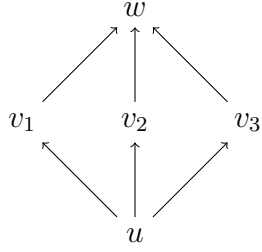
\begin{center}
        \tikz {
            \node (a) at (0,0) {$u$};
            \node (b1) at (-1.5,1.5) {$v_1$};
            \node (c) at (0,3) {$w$};
            \node (b3) at (1.5,1.5) {$v_3$};
            \node (b2) at (0,1.5) {$v_2$};
            \draw[->] (a) -- (b1);
            \draw[->] (b1) -- (c);
            \draw[->] (b2) -- (c);
            \draw[->] (a) -- (b2);
            \draw[->] (b3) -- (c);
            \draw[->] (a) -- (b3);
        }
    \end{center}
    Up to a choice of signs, there are exactly three $w$-complete, connected upper face multigraphs that can be constructed out of paths $e_{u,v_1}, e_{u,v_2}, e_{u,v_3} \in \Omega_{1}(G;R)$, namely
    \begin{center}
        \tikz{
            \node (b1) at (0,0) {$e_{u,v_1}$};
            \node (b2) at (3,0) {$- e_{u, v_2}$};
            \draw[-] (b1) -- (b2) node[midway,above] {$e_u$};
        }
        \;\;\;
        \tikz {
            \node (b1) at (0,0) {$e_{u,v_1}$};
            \node (b2) at (3,0) {$- e_{u, v_3}$};
            \draw[-] (b1) -- (b2) node[midway,above] {$e_u$};
        }
        \;\;\;
        \tikz {
            \node (b1) at (0,0) {$e_{u,v_2}$};
            \node (b2) at (3,0) {$- e_{u, v_3}$};
            \draw[-] (b1) -- (b2) node[midway,above] {$e_u$};
        }
    \end{center}
    When $R=\mathbb{Z}_3$ there is in addition the following connected upper face multihypergraph.
    
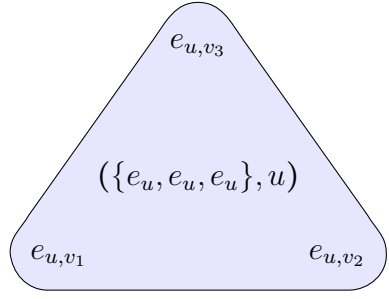
\begin{center}
    \begin{tikzpicture}
        \node (v1) at (0,0) {};
        \node (v2) at (4,0) {};
        \node (v3) at (2,2.828) {};
        \filldraw[fill=blue!10] ($(v1)+(-0.25,0.5)$)
            -- ($(v3)+(-0.6,0)$)
            to[out=55,in=180] ($(v3)+(0,0.5)$)
            to[out=0,in=125] ($(v3)+(0.6,0)$)
            -- ($(v2)+(0.25,0.5)$)
            to[out=300,in=90] ($(v2)+(0.5,0)$)
            to[out=275,in=0] ($(v2)+(0,-0.5)$)
            -- ($(v1)+(0,-0.5)$)
            to[out=180,in=275] ($(v1)+(-0.5,0)$)
            to[out=90,in=240] ($(v1)+(-0.25,0.5)$);
        \node (l) at (2,1) {$(\{e_u,e_u,e_u\},u)$};
        \node (v-1) at (0.15,-0.05) {$e_{u,v_1}$};
        \node (v-2) at (3.85,-0.05) {$e_{u,v_2}$};
        \node (v-3) at (2,2.75) {$e_{u,v_3}$};
    \end{tikzpicture}
    \end{center}
    However, $\dim\Omega_2(G;\mathbb{Z}_3)=2$, and a basis for $\Omega_2(G;\mathbb{Z}_3)$ can be obtained by forming extensions over any two of the face multigraphs above.
    Digraphs demonstrating the necessity of allowing face multihypergraphs with hyperedges of size greater than $2$ in forming a basis are given in Example~\ref{exam:DiffEulerOverFiniteFields}.
\end{exmp}

An \emph{inclusion} of face multihypergraphs is an inclusion of multihypergraphs that preserves 
vertex and edge labels.
A face multihypergraph $F_1$ is a \emph{sub-face multihypergraph} of a face multihypergraph $F_2$ if there is an inclusion of face multihypergraph from $F_1$ to $F_2$.

\begin{exmp}\label{ex:CompleteMaximal}
    For each $v \in V_G$, by definition, any connected $v$-complete face multihypergraph is maximal among all connected $v$-proper face multihypergraph when these objects are partially ordered by inclusion of sub-face multihypergraphs.
\end{exmp}

A pair $(x,v)$ for $x\in \Omega_n(G;R)$ and $v\in V_G$ is called \emph{upper isolated} if for any $u\in \bar{h}(\delta_{n-1}^h\delta_{n,v}^h(x))$ and $w_h\in \bar{h}(x)$, either $u=w_h$ or there is an edge $(u,w_h)\in E_G$.
Similarly, a pair $(x,v)$ for $x\in \Omega_n(G;R)$ and $v\in V_G$ is called \emph{lower isolated} if for any $u\in \bar{t}(\delta^t_{n-1}\delta^t_{n,v}(x))$ and $w_t\in \bar{t}(x)$, either $u=w_t$ or there is an edge $(w_t,u)\in E_G$.
A digraph $G$ is called \emph{transitive} when for any $u,v,w \in V$, if $(u,v),(v,w) \in E_G$ then $(u,w)\in E_G$. 

\begin{exmp}
    A connected $v$-proper upper (lower) face multihypergraph containing a vertex $x$ such that $(x,v)$ is upper (lower) isolated consists of only that vertex. 
    In particular, if $G$ is a transitive digraph, then all connected $v$-proper upper (lower) face multihypergraphs are of this form.
\end{exmp}

\begin{exmp}\label{exam:Dimension3}
    Recall from Proposition~\ref{prop:Dim2Base} that directed squares of the form
    \[
        S_{u,v_1,v_2,w}= e_{u,v_1,w} - e_{u,v_2,w} \in \Omega_2(G;R)
    \]
    such that $(u,w) \notin G$, $u \neq w$ and $v_1 \neq v_2$, directed triangles
    \[
        T_{u,v,w}= e_{u,v,w} \in \Omega_2(G;R)
    \]
    such that $(u,w) \in G$, and double edges
    \[
        E_{u,v} = e_{uvu} \in \Omega_2(G;R)
    \]
    for some $u,v,v_1,v_2,w \in V_G$
    generate $\Omega_2(G;\mathbb{Z})$.
    
    For any $v'\in V_G$, the only nonzero $\delta_{n,v'}^h(T_{u,v,w})$ and $\delta_{n,v'}^h(E_{u,v})$ occur when $v'=v$, in which case
    \[
        \delta_{n,v}^h(T_{u,v,w}) = e_{u,v}
        \;\;\; \text{and} \;\;\;
        \delta_{n,v}^h(E_{u,v}) = e_{u,v}.
    \]
    Hence, if $T_{u,v,w}$ or $E_{u,v}$ appears as a vertex in an upper face multigraph, then it has valence at most $1$.
    Similarly, if $S_{u,v_1,v_2,w}$ appears as a vertex in an upper face multigraph, then it has valence at most $2$.
    Therefore, any upper face multigraph with vertices, double edges, directed triangles, or directed squares in $\Omega_3(G;R)$ consists of a disjoint union of face multigraphs that are either lines or cycle as unlabeled multigraphs.
    
    Clearly, an analogous statement also holds for lower face multigraphs, and for lower or upper face multihypergraphs when $G$ has no multisquares or when $R$ is an abelian group with no additive torsion of order greater than $2$.
\end{exmp}

\subsection{Existence of extensions over face multihypergraphs}

The next proposition demonstrates the necessity of considering vertex complete extensions over face multihypergraphs.

\begin{proposition}\label{prop:GraphExtentsion}
    Let $F_G^n(x_1,\dots,x_m)$ be a face multihypergraph, where $x_1,\dots,x_m\in \Omega_n(G;R)$, and let $v \in V_G$. If the upper extension of $x_1+\cdots+x_m$ by $v$ is $v$-complete over $F_G^n(x_1,\dots,x_m)$, 
    then
    \[
        [x_1 + \cdots + x_m]^v \in \Omega_{n+1}(G;R).
    \]
    Similarly, if the lower extension of $x_1+\cdots+x_m$ by $v$ is $v$-complete over $F_G^n(x_1,\dots,x_m)$, then
    \[
        [x_1 + \cdots + x_m]_v \in \Omega_{n+1}(G;R).
    \]
\end{proposition}

\begin{proof}
    We prove the case of upper extensions, the proof for lower extensions being similar.
    As $x_1+\cdots+x_m \in \Omega_n(G)$, by Lemma~\ref{lem:NoPositionSwap} we have that $\partial^M_{n,n,i}(x_1+\cdots+x_m) = 0$ and so $\partial^M_{n+1,n+1,i}([x_1+\cdots+x_m]^v) = 0$ for each $i=1,\dots,n-1$.
    Therefore, it remains to check that 
    \[
        \partial^M_{n+1,n+1,n}([x_1+\cdots+x_m]^v) = 0.
    \]
    
    To begin, consider any nonzero path $e_{u_0,\dots,u_{n-1}}$ occurring as a nonzero summand $\alpha e_{u_0,\dots,u_{n-1}}$ in some $x_i^{u_{n-1},k}$ for some $i = 1 ,\dots, m$, $k=1,\dots,m_i$ from the sum decomposition in equation~\eqref{eq:n-1Faces} of Definition~\ref{def:FaceMultiGraph}.
    If 
    $(u_{n-1},v) \in E_G$ or $u_{n-1} = v$,
    then by definition
    \begin{equation}\label{eq:DegenerateCase}
         \partial^M_{n+1,n+1,n}([[e_{u_0,\dots,u_{n-1}}]^{u_n}]^v )
         = 0
    \end{equation}
    for any $u_{n}\in V_G$. Thus we reduce to the case that  $(u_{n-1},v) \notin E_G$ and $u_{n-1} \neq v$.
    
    We consider first the case of edges of $F_G^n(x_1,\dots,x_m)$ and then hyperedges of size greater than $2$.
    When
    $(u_{n-1},v) \notin E_G$, $u_{n-1} \neq v$ and $x_i^{u_{n-1},k}$ appears uniquely in an edge of $F_G^n(x_1,\dots,x_m)$,
    by the completeness of $F_G^n(x_1,\dots,x_m)$, the unique edge associated to $x_i^{u_{n-1},k}$ provides a unique $j\in \{1,\dots,m\}$ such that $j\neq i$ and $l\in \{1,\dots,m_j\}$ with
    \[
        x_i^{u_{n-1},k}=-x_j^{u_{n-1},l}.
    \]
    In particular, $-\alpha e_{u_0,\dots,u_{n-1}}$ occurs as a summand in $x_j^{u_{n-1},l}$ uniquely corresponding to summand $\alpha e_{u_0,\dots,u_{n-1}, u_j}$ in $x_{i}$.
    Moreover,
    \begin{equation}\label{eq:Non-DegenerateCase}
        \partial^M_{n+1,n+1,n}(
        [[e_{u_0,\dots,u_{n-1}}]^{u_n}]^v
        +
        [[-e_{u_0,\dots,u_{n-1}}]^{u'_n}]^v
        )
        = 0
    \end{equation}
    for any $u_{n}, u'_n\in V_G$ such that $(u_{n-1},u_n),(u_{n-1},u'_n),(u_n,v),(u'_n,v)\in E_G$.
    
    We now return to the remaining case of a hyperedge of size greater than $2$.
    Thus suppose $(u_{n-1},v) \notin E_G$, $u_{n-1} \neq v$ and $x_i^{u_{n-1},k}$ appears uniquely in a hyperedge of $F_G^n(x_1,\dots,x_m)$ of size greater than $2$. Then by the completeness of $F_G^n(x_1,\dots,x_m)$, the unique edge associated to $x_i^{u_{n-1},k}$ of size $t+1 \geq 3$ provides unique $i_1,\dots,i_t \in \{ 1,\dots,m \}$ such that $i \neq i_{j'}$ for some $j' = 1,\dots,t$ with
    \[
        x_i^{u_{n-1},k}=x_{i_j}^{u_{n-1},k_j}
        \;\;\; \text{and} \;\;\;
        x_i^{u_{n-1},k} + x_{i_1}^{u_{n-1},k_1} + \cdots + x_{i_t}^{u_{n-1},k_t} = 0
    \]
    for each $j=1,\dots,t$, where $k_j \in \{1,\dots,m_{i_j}^u\}$ is as in Definition~\ref{def:FaceMultihyperGraph}.
    In particular, $\alpha e_{u_0,\dots,u_{n-1}}$ occurs as a summand in each $x_{i_j}^{u_{n-1},k_j}$ uniquely corresponding to a summand $\alpha e_{u_0,\dots,u_{n-1}, u_{i_j}}$ in $x_{i}$ for each $j=1,\dots,t$.
    Moreover,
    \begin{equation}\label{eq:Non-DegenerateCase2}
        \partial^M_{n+1,n+1,n}(
        [[e_{u_0,\dots,u_{n-1}}]^{u_n}]^v
        +
        [[e_{u_0,\dots,u_{n-1}}]^{u^1_n}]^v
        + \cdots +
        [[e_{u_0,\dots,u_{n-1}}]^{u^t_n}]^v
        )
        = 0
    \end{equation}
    for any $u_{n}, u^{1}_n,\dotsm, u^{t}_n\in V_G$ such that 
    \[
        (u_{n-1},u_n),(u_{n-1},u^1_n),\dots,(u_{n-1},u^t_n),(u_n,v),(u_n^1,v),\dots,(u_n^t,v)\in E_G.
    \]\vspace{-0.5cm}
    
    We note that it is possible that the path $e_{u_0,\dots,u_{n-1}}$ does not upper extend by some $u_n \in V_G$ and correspond to a path in a nonzero summand of $x_i$.
    However, all such paths cancel within the decomposition $x_i^{u_{n-1},1}+\cdots+x_i^{u_{n-1},m_i} = \delta_{n,u_{n-1}}^h(x_i)$ prior to upper extension.
    Otherwise, equations~\eqref{eq:Non-DegenerateCase} and \eqref{eq:Non-DegenerateCase2} together with the unique determination of $j$, $i_j$ and
    the case corresponding to equation~\eqref{eq:DegenerateCase}, imply that $\partial^M_{n+1,n+1,n}([x^{n-1}_1+\cdots+x^{n-1}_m]^v) = 0$ as required.
    \qed
\end{proof}
We remark that the proof of Proposition~\ref{prop:GraphExtentsion} for the special case of face multigraphs alone is somewhat simpler.

\subsection{Connectedness and mutations of face multihypergraphs}\label{sec:Connectedness}

We now give a variant notion of connectedness for face multihypergraphs which is stronger than that of the connectedness of the underlying multihypergraph.
This is done initially in the case of face multigraphs, and then for face multihypergraphs.

The incentive for introducing this notation of connectedness will become apparent in the next section while trying to reduce the size of the generating sets for $\Omega_n(G;R)$. There are usually many face multihypergraphs which can be extended to the same element, and restricting to
mutation equivalence classes
is a way to reduce the number under consideration.

\begin{definition}\label{def:MutationMultigraph}
    A face multigraph $F_G^n(x_1,\dots,x_m)$ is called a \emph{mutation} of a face multigraph $\bar{F}_G^n(x_1,\dots,x_m)$ on $x_1,\dots,x_m\in \Omega_n(G;R)$ if $F_G^n(x_1,\dots,x_m)$ can be obtained from $\bar{F}_G^n(x_1,\dots,x_m)$ by replacing two edges $(\{x_i^{u,k},x_j^{u,l}\},u)$ and $(\{x_{i'}^{u,k'},x_{j'}^{u,l'}\},u)$ by edges
    \begin{align*}
    (\{x_i^{u,k},x_{j'}^{u,l'}\},u) \: &\text{and} \: (\{x_j^{u,l},x_{i'}^{u,k'}\},u),
    \: \text{or} \\
    (\{x_i^{u,k},x_{i'}^{u,k'}\},u) \: &\text{and} \: (\{x_j^{u,l},x_{j'}^{u,l'}\},u).
    \end{align*}
    Mutations are symmetric by construction and generate an equivalence relation on face multigraphs $F_G^n(x_1,\dots,x_m)$ with $x_1,\dots,x_m\in \Omega_n(G;R)$.
    We therefore call two face multigraphs \emph{mutation equivalent} if they are obtainable from each other by a sequence of mutations.
\end{definition}

By construction, mutations preserve properness and completeness of face multigraphs.
Furthermore, upper (lower) extensions by a vertex $v\in V_G$ over a $v$-complete face multigraph $F^n_G(x_1,\dots,x_m)$ are invariant under mutation equivalence of the face multigraph as the extension only depends on the choice of $x_1,\dots,x_m \in \Omega_n(G;R)$ and $v\in V_G$.

\begin{exmp}\label{exam:2Trapezohedron}
    Consider the trapezohedron $\mathbb{T}_2$ of order $2$ from Definition~\ref{def:Trapezohedron}.
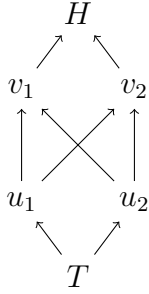
    \begin{center}
        \tikz {
            \node (T) at (0,0) {$T$};
            \node (a1) at (-0.75,1) {$u_1$};
            \node (a2) at (0.75,1) {$u_2$};
            \node (b1) at (-0.75,2.5) {$v_1$};
            \node (b2) at (0.75,2.5) {$v_2$};
            \node (H) at (0,3.5) {$H$};
            \draw[->] (T) -- (a1);
            \draw[->] (T) -- (a2);
            \draw[->] (a1) -- (b1);
            \draw[->] (a1) -- (b2);
            \draw[->] (a2) -- (b1);
            \draw[->] (a2) -- (b2);
            \draw[->] (b1) -- (H);
            \draw[->] (b2) -- (H);
        }
    \end{center}
    There are up to sign two directed squares 
    \begin{align*}
        S_1 &=
        e_{T,u_1,v_1}
        -
        e_{T,u_2,v_1}
        \in \Omega_2(\mathbb{T}_2;R)
        \\
        S_2 &=
        e_{T,u_1,v_2}
        -
        e_{T,u_2,v_2}
        \in \Omega_2(\mathbb{T}_2;R)
    \end{align*}
    whose tail vertices are $T$.
    The upper extension $[S_1-S_2]^H$ is an upper extension over the $H$-complete face multigraph
    
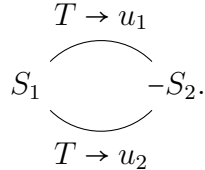
\begin{center}
        \tikz {
            \node (S1) at (0,0) {$S_1$};
            \node (S2) at (2,0) {$-S_2.$};
            \draw[-] (S1) to [out=45,in=135] node[pos=0.5,above] {$T\to u_1$} (S2);
            \draw[-] (S1) to [out=315,in=225] node[pos=0.5,below] {$T\to u_2$} (S2);
        }
    \end{center}
    Hence, $[S_1-S_2]^H \in \Omega_3(G;R)$ by Proposition~\ref{prop:GraphExtentsion} and coincides with the trapezohedron element generating $\Omega_3(\mathbb{T}_2;R)$.
    
    On the other hand, consider the following two H-complete face multigraphs
    
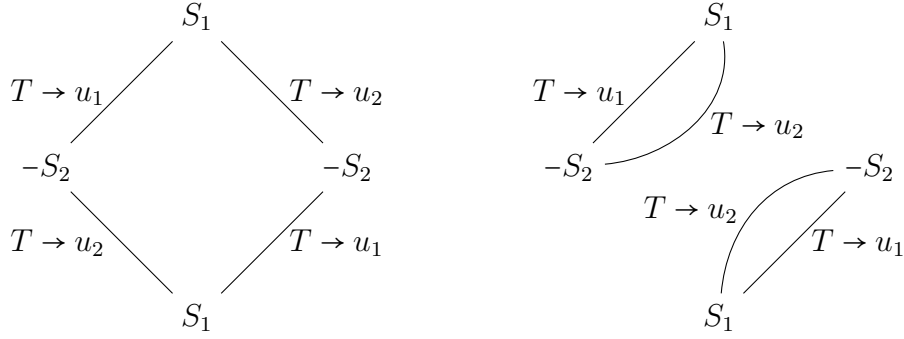
\begin{center}
        \tikz {
            \node (S1) at (0,0) {$S_1$};
            \node (S2) at (2,2) {$-S_2$};
            \node (S-1) at (0,4) {$S_1$};
            \node (S-2) at (-2,2) {$-S_2$};
            \draw[-] (S1) -- (S2)  node[midway,right] {$\:T\to u_1$};
            \draw[-] (S1) -- (S-2)  node[midway,left] {$T\to u_2\:$};
            \draw[-] (S2) -- (S-1)  node[midway,right] {$\:T\to u_2$};
            \draw[-] (S-2) -- (S-1) node[midway,left] {$T\to u_1\:$};
        }
        \;\;\;\;\;\;\;\;\;\;\;\;
        \tikz {
            \node (S1) at (0,0) {$S_1$};
            \node (S2) at (2,2) {$-S_2$};
            \node (S-1) at (0,4) {$S_1$};
            \node (S-2) at (-2,2) {$-S_2$};
            \draw[-] (S1) -- (S2)  node[midway,right] {$\:T\to u_1$};
            \draw[-] (S2) to [out=185,in=85] node[pos=0.5,left] {$T\to u_2\:$} (S1);
            \draw[-] (S-1) to [out=280,in=5] node[pos=0.5,right] {$\:T\to u_2$} (S-2);
            \draw[-] (S-2) -- (S-1)  node[midway,left] {$T\to u_1\:$};
        }
    \end{center}
    and note that they are mutation equivalent. The element $[S_1+S_1-S_2-S_2]^H$ upper extends over the left-hand face multigraph, and disconnectedness of the right-hand face multigraph gives
    \[
        [S_1+S_1-S_2-S_2]^H = 2 [S_1-S_2]^H.
    \]
    
    Hence, for example, $[S_1+S_1-S_2-S_2]^H$ cannot be an element of a basis of $\Omega_3(\mathbb{T}_2;\mathbb{Z})$.
    More generally, the unique generator up to sign of $\Omega_3(\mathbb{T}_t;\mathbb{Z})$ for $t\geq 3$ (c.f. Definition \ref{def:Trapezohedron}) can be similarly obtained as an extension over a face multigraph whose underlying multigraph is a cycle on $t$ vertices.
\end{exmp}

The previous example demonstrates that within mutation equivalence classes, connectedness of the underlying multigraphs need not be preserved.
Therefore, we require a stronger notion of connectedness for mutation equivalence classes of face multigraphs. 

\begin{definition}\label{def:connectedeness}
    A face multigraph is called \emph{strongly connected} if it and all its mutation equivalent multigraphs are connected as multigraphs.
\end{definition}

The definitions given above are only applicable to face multigraphs. More complicated definitions must be formulated to cover the more general case of face multihypergraphs, and we turn now to this task.

\begin{definition}\label{def:MutationMultihypergraph}
    Let $x_1,\dots,x_m\in \Omega_n(G;R)$. A face multihypergraph $F_G^n(x_1,\dots,x_m)$ is called a \emph{mutation} of a face multihypergraph $\bar{F}_G^n(x_1,\dots,x_m)$ if it can be obtained from $\bar{F}_G^n(x_1,\dots,x_m)$ by any of the following operations;
    \begin{enumerate}
        \item a mutation on edges of the form in Definition~\ref{def:MutationMultigraph},
        \item replacing two hyperedges
    \[
        (\{x_{i_1}^{u,k_1},\dots,x_{i_t}^{u,k_t}\},u)
        \;\;\; \text{and} \;\;\;
        (\{x_{i'_1}^{u,k'_1},\dots,x_{i'_t}^{u,k'_t}\},u)
    \]
    for $t\geq 3$ by edges
    \begin{align*}
        (\{x_{i_j}^{u,k_j},x_{i'_j}^{u,k'_j}\},u)
    \end{align*}
    for each $j=1,\dots,t$, or by the reverse replacement,
    \item replacing two hyperedges
    \[
        (\{x_{i_1}^{u,k_1},\dots,x_{i_t}^{u,k_t}\},u)
        \;\;\; \text{and} \;\;\;
        (\{x_{i_{t+1}}^{u,k_{t+1}},\dots,x_{i_{2t}}^{u,k_{2t}}\},u)
    \]
    for $t\geq 3$ by hyperedges
    \begin{align*}
        (\{x_{i_{j_1}}^{u,k_{j_1}},\dots,x_{i_{j_t}}^{u,k_{j_t}}\},u)
        \;\;\; \text{and} \;\;\;
        (\{x_{i_{j_{t+1}}}^{u,k_{j_{t+1}}},\dots,x_{i_{j_{2t}}}^{u,k_{j_{2t}}}\},u)
    \end{align*}
    for some permutation $j_1,\dots,j_{2t}$ of $1,\dots,2t$.
    \end{enumerate}
    Mutations generate an equivalence relation on face multihypergraphs $F_G^n(x_1,\dots,x_m)$ with $x_1,\dots,x_m\in \Omega_n(G;R)$.
    We call two face multihypergraphs mutation \emph{equivalent} if they are obtainable from each other by a sequence of mutations.
\end{definition}

\begin{definition}\label{def:connectedeness2}
    A face multihypergraph is called \emph{strongly-connected} if it and all its mutation equivalent multihypergraphs are connected as multihypergraphs.
\end{definition}

From the perspective of obtaining a generating set for $\Omega_n(G;R)$, it is important to note that the existence of a face multihypergraph with hyperedges of size greater than $2$ can increase the number of possible extensions in a given degree. Moreover, this can occur even up to mutation equivalence and can produce additional, linearly dependent elements. Such a situation is highlighted in the next example.

\begin{exmp}
    The face multihypergraph appearing in Example~\ref{ex:SimpleMultiedge} is not mutation equivalent to any of the other multihypergraphs appearing in the same example.
\end{exmp}

\section{Inductive elements \texorpdfstring{$\Omega_n(G;R)$}{path chains} generating sets and bases}\label{sec:InductiveBasis}

We are now ready to define our main objects of study and prove the central results demonstrating their use.
Throughout this section, let $G$ be a digraph, $n$ a non-negative integer, $R$ a commutative ring with unit, and $\mathbb{Z}_p$ (for a prime number $p$) the finite field with $p$ elements, unless stated otherwise.

\subsection{Inductively extending bases}

In this subsection we show how to inductively define elements of $\Omega_n(G;R)$ in terms of arbitrary bases in the previous two dimensions. In the next subsection we set out the important special case in which we are primarily interested.
When considering a non-finite $G$ in this section, we note that the axiom of choice is required in the construction of bases as a subset of infinite spanning sets of a vector space. 

\begin{definition}\label{def:InductiveElementOverBases}
    Let $n\geq 1$ and suppose that $E_{n-1} \subseteq \Omega_{n-1}(G;R)$ and  $E_{n-2}\subseteq\Omega_{n-2}(G;R)$ are subsets whose elements respect the bigradings within $\Omega_{n-1}^{*,*}(G;R)$ and $\Omega_{n-2}^{*,*}(G;R)$, respectively.
    An upper (lower) connected element $x \in \Omega_n(G,R)$ is called \emph{upper (lower) $(E_{n-1},E_{n-2})$-inductive} if it can be obtained as an upper (lower) extension over a strongly connected $h(x)$-complete ($t(x)$-complete) face multihypergraph $F^{n-1}_G(x_1,\dots,x_m)$ satisfying the following conditions.
    \begin{enumerate}
        \item 
        Each $x_1,\dots,x_m$ is up to sign an element of $E_{n-1}$.
        \item 
        The elements $x_i^{u_{n-1},k}$ for $i = 1,\dots, m$, and $k=1,\dots,m_i$ which appear in the sum decomposition~\eqref{eq:n-1Faces} are elements of $E_{n-2}$ up to sign.
    \end{enumerate}
    A face multihypergraph $F^{n-1}_G(x_1,\dots,x_m)$ of the form above is called an \emph{upper (lower)\linebreak $(E_{n-1},E_{n-2})$-inductive} face multihypergraph of $x$. 
    Given an upper (lower) $(E_{n-1},E_{n-2})$-inductive element $x\in \Omega_n(G;R)$, we call the choice of an upper (lower) $(E_{n-1},E_{n-2})$-inductive face multihypergraph of $x$ 
    an \emph{upper (lower) $(E_{n-1},E_{n-2})$-inductive structure on $x$}.
\end{definition}

\begin{theorem}\label{thm:BasisExtension}
    Let $n\geq 1$, $R=\mathbb{Z}$ or $R=\mathbb{Z}_p$ for a prime $p$, $B_{n-1}$ be an $R$-basis of $\Omega_{n-1}(G;R)$, and $B_{n-2}$ an $R$-basis of $\Omega_{n-2}(G;R)$, both of which respect the bigradings within $\Omega_{n-1}^{*,*}(G;R)$ and $\Omega_{n-2}^{*,*}(G;R)$, respectively.
    Then the upper (lower) $(B_1,B_2)$-inductive elements generate $\Omega_n(G;R)$.
    Moreover,
    we have the following conditions on 
    $(B_1,B_2)$-inductive elements containing a basis.
    \begin{enumerate}
        \item
        When $R=\mathbb{Z}_p$,
        a choice of $\mathbb{Z}_p$-basis always exists as a subset of $(B_{n-1},B_{n-1})$-inductive elements.
        \item
        Let $K$ be a field of characteristic $0$.
        Then a subset of the image of $(B_{n-1},B_{n-1})$-inductive elements under
        \[
            \Omega_{n}(G;\mathbb{Z}) \hookrightarrow \Omega_{n}(G;\mathbb{Z}) \otimes K 
            \xrightarrow{\mu_n}
            \Omega_{n}(G;K),
        \]
        where $\mu_n$ is the isomorphism from Lemma~\ref{lem:CoeffChange}, form a basis of $\Omega_{n}(G;K)$.
    \end{enumerate}
\end{theorem}

\begin{proof}
    We prove the case of upper inductive, the proof for lower inductive being similar.
    First we prove that the $(B_{n-1},B_{n-2})$-inductive elements
    generate $\Omega_n(G;R)$ when $R=\mathbb{Z}$ or $R=\mathbb{Z}_p$.
    To this end take a basis $B_n$ of $\Omega_n(G;R)$ that respects the bigrading $\Omega_n^{*,*}(G;R)$ and
    let $x\in B_n$.
    As $B_n$ respects the bigrading $\Omega_n^{*,*}(G;R)$, $x$ is upper connected and the vertex $h(x)\in V_G$ is well defined. 
    Our aim is to construct an upper connected $h(x)$-complete $(B_{n-1},B_{n-2})$-inductive face multihypergraph $F^{n-1}_G(x_1,\dots,x_m)$
    such that $x$ is the upper extension by $h(x)$ over $F^{n-1}_G(x_1,\dots,x_m)$ and so
    \[
        x = [x_1+\cdots+x_m]^{h(x)}.
    \]
    
    For any $v\in V$, we have that $\delta_{n,v}^h(x) \in \Omega_{n-1}(G;R)$ by Lemma~\ref{lem:SumFace}. Assuming $\delta_{n,v}^h(x)\neq 0$, then as $B_{n-1}$ is an $R$-basis, there is a smallest integer $t_v >0$ with a unique up to reordering sum decomposition
    \begin{equation}\label{eq:n-1MinimalDecompostion}
        \delta_{n,v}^h(x) = x_1^v+\cdots+x_{t_v}^v
    \end{equation}
    where either $x_k^v \in B_{n-1}$ or $-x_k^v \in B_{n-1}$ for $k=1,\dots,t_v$. Here it is important to realise that we allow for repetition of the $x_i^v$'s and expand
    \[
    k\cdot y=\begin{cases}y+y+\cdots+y\; (\text{$k$ times})&k\geq0\\
    -y-y-\cdots-y\; (\text{$-k$ times})&k<0\end{cases}
    \]
    for a (mod $p$) integer $k$.
    
    In any case, we may take the multiset $x_1,\dots,x_m$ to be an enumeration of the elements in the finite multiset
    \begin{align*}
        \{ x_k^v \: | \: &
        v\in V_G, \:
        \delta_{n,v}^h(x)\neq 0, \:
        \text{and} \:
        \text{$x_k^v$ appears in decomposition~\eqref{eq:n-1MinimalDecompostion} for} \: k=1,\dots,t_v  \}
    \end{align*}
    Note that we cannot have $x_i = -x_j$ for any $i,j=1,\dots,m$, as this would imply that paths in $[x_i]^{h(x)}$ and $[x_j]^{h(x)}$ cancel, which is impossible as $x \in B_n \subseteq \Omega_n(G;R)$.
    
    Now, for any $v\in V$ and $i=1,\dots,m$, we have that $\delta_{n,v}^h(x_i) \in \Omega_{n-2}(G;R)$ by Lemma~\ref{lem:SumFace}.
    Suppose that $i= 1,\dots,m$, $\delta_{n,v}^h(x_i) \neq 0$, $(v,h(x_i)) \notin E_G$ and $v \neq h(x_i)$.
    Then, similarly to equation~\eqref{eq:n-1MinimalDecompostion},
    there is a smallest integer $t_{v}^i >0$ with a unique up to reordering sum decomposition
    \begin{equation}\label{eq:n-2MinimalDecompostion}
        \delta^h_{n,v}(x_i) = x_1^{v,i}+\cdots+x_{t_v^i}^{v,i}
    \end{equation}
    where either $x_k^{v,i} \in B_{n-2}$ or $-x_k^{v,i} \in B_{n-2}$ for $k=1,\dots,t_v^i$.
    Again, as previously
    \begin{equation}\label{eq:n-2Idependence}
        x_k^{v,i} \neq -x_{k'}^{v,i}
    \end{equation}
    for any $k,k'=1,\dots,t_v^i$.
    Moreover, no sub-sequence of $x_1^{v,i},\dots,x_{t_v^i}^{v,i}$ can sum to zero as
    equation~\eqref{eq:n-2Idependence} is obtained as a refinement of a unique linear combination in the members of the $R$-basis $B_{n-2}$.
    
    Let
    \[
        V_x =
        \{ v \in V_G \: | \:
        \delta_{n,v}^h(x_i) \neq 0, \:
        (v,h(x))\notin E_G \:
        \text{and} \:
        v \neq h(x) \:
        \text{for some} \: i = 1,\dots,m \}.
    \]
    As $x \in B_n \subseteq \Omega_n(G;R)$, by Lemma~\ref{lem:NoPositionSwap}, $\partial^M_{n,n,j}(x)=0$ for each $j=1,\dots,n-1$ and hence $\partial^M_{n-1,n-1,j'}\delta^h_{n,v}(x) = 0$ for each $v\in V_G$, and $j'=1,\dots,n-2$.
    Furthermore,
    \[
        0
        =
        \partial_{n-1,n-1}^M\delta^h_{n,v}(x) 
        =
        \sum_{v \in V_x} \sum_{k=1}^{t_v^i} x_k^{v,i}.
    \]
    
    Hence when $R=\mathbb{Z}$ or $R=\mathbb{Z}_2$, using the fact that each $x_k^{v,i}$ is up to sign an element of $B_{n-2}$ and that any additive torsion in $R$ of order $2$, any $x_k^{v,i}$ can be paired with another $x_{k'}^{v,i'}$ such that
    \begin{equation}\label{eq:OpositSignPairing}
        x_k^{v,i} = -x_{k'}^{v,i'}
    \end{equation}
    for some $i'=1,\dots,m$ and $k'=1,\dots,t_v^{i'}$.
    In addition, as a consequence of equation~\eqref{eq:n-2Idependence}, it must be the case that $i\neq i'$.
    
   When $R=\mathbb{Z}_p$ for $p\geq 3$, not all $x_k^{v,i}$ can necessarily be paired in the form of equation~\eqref{eq:OpositSignPairing}, and there may be additional matchings of the form
    \begin{equation*}
        x_k^{v,i} = x_{k_1}^{v,i_1} = \cdots = x_{k_{p-1}}^{v,i_{p-1}}
        \;\;\;\text{such that} \;\;\;
        x_k^{v,i} + x_{k_1}^{v,i_1} + \cdots + x_{k_{p-1}}^{v,i_{p-1}} = 0
    \end{equation*}
    for some $i_j \in \{1,\dots,m \}$ and $k_j\in \{1,\dots,t_v^{i_j}\}$.
    In addition, as no sub-sequence of $x_1^{v,i},\dots,x_{t_v^i}^{v,i}$ from equation~\eqref{eq:n-2MinimalDecompostion} can sum to zero, it must be the case that $k\neq k_j$ for some $j=1,\dots,p-1$.
    
    We may now construct a face multihypergraph $F^{n-1}_G(x_1,\dots,x_m)$ on vertices $x_1,\dots,x_m$ with all edges and hyperedges described by the labels
    \[
        (\{ x_k^{v,i}, x_{k'}^{v,i'} \}, v)
        \;\;\ \text{and} \;\;\;
        (\{  x_k^{v,i}, x_{k_1}^{v,i_1}, \dots, x_{k_{p-1}}^{v,i_{p-1}} \}, v)
    \]
    using the pairings and matchings obtained above.
    The face multihypergraph $F^{n-1}_G(x_1,\dots,x_m)$ is $h(x)$-complete as we have constructed all required hyperedges for vertices in $V_x$.
    However, the face multihypergraph $F_G^{n-1}(x_1,\dots,x_m)$ need not be strongly connected.
    Nevertheless, a minimal subdivision of $F_G^{n-1}(x_1,\dots,x_m)$ into strongly connected face multihypergraph provides $(B_{n-1},B_{n-2})$-inductive elements whose sum is $x$.
    Therefore, every basis element in $B_n$ can be written as a linear combination of $(B_{n-1},B_{n-2})$-inductive elements and so $(B_{n-1},B_{n-2})$-inductive elements generate $\Omega_n(G;R)$ as required.

    For part (1), when $R = \mathbb{Z}_p$, a generating set is a spanning set of a vectors space and can always be reduced to a basis.
    Finally, part (2) follows from the fact that $(B_{n-1},B_{n-2})$-inductive elements generate $\Omega_n(G;\mathbb{Z})$
    combined with Lemma~\ref{lem:CoeffChange} and again that a spanning sets of a vector space always reduce to a basis.
    \qed
\end{proof}

\subsection{Inductive elements}

In this subsection we describe an inductive method for constructing generating sets of the path chain modules $\Omega_n(G;R)$. The heavy lifting here is done by Theorem~\ref{thm:BasisExtension}. The main theorems included in the introduction are stated with greater generality in this subsection and proved in full.

\begin{definition}\label{def:InductiveElement}
    Define sets $\bar{E}_{n}^h,\bar{E}_{n}^t \subseteq \Omega_n(G;R)$ for $n\geq-1$ by means of the following inductive construction.
    \begin{enumerate}
        \item 
        Let $\bar{E}^h_{-1} = \bar{E}^t_{-1} = \emptyset$
        and $\bar{E}^h_0 = \bar{E}^t_0 = \{\pm e_v \: | \: v \in V_G \}$.
        \item 
        For $n\geq 1$, the set $\bar{E}^h_n$ ($\bar{E}^t_n$) consists of all upper (lower) $(\bar{E}^h_{n-1},\bar{E}^h_{n-2})$-inductive elements ($(\bar{E}^t_{n-1},\bar{E}^t_{n-2})$-inductive elements).
    \end{enumerate}
    An element belonging to $\bar{E}_n^h$ ($\bar{E}^t_n$) for some $n\geq 0$ is called an
    \emph{upper (lower) inductive element}.
    The members of $\bar{E}^h_n$ ($\bar{E}^t_n$) for a fixed $n$ are called \emph{$n$-dimensional upper (lower)
    inductive elements}.
\end{definition}
Note that, by construction, all upper and lower inductive elements are connected. In particular, $h(x)$ and $t(x)$ are well defined for any upper or lower inductive element $x$. 

Denote the submodule of $\Omega_n(G;R)$ generated by $\bar{E}^h_n$ ($\bar{E}^t_n$) as
\[
    \Omega_n^{I,h}(G;R) \subseteq \Omega_n(G;R)
    \;\;\;
    \left( \Omega_n^{I,t}(G;R) \subseteq \Omega_n(G;R) \right).
\]
An $(\bar{E}^h_{n-1},\bar{E}^h_{n-2})$-face multihypergraph ($(\bar{E}^t_{n-1},\bar{E}^t_{n-2})$-face multihypergraph) is called an \emph{upper (lower) inductive face multihypergraph}. 
Given an upper (lower) inductive $x\in \Omega_n(G;R)$, we call a choice of upper (lower) inductive face multihypergraph for which $x$ is an upper (lower) extension by $h(x)$ ($t(x)$) an \emph{upper (lower) inductive structure on $x$}. For all this terminology, to refer to either the upper or lower inductive cases, we simply write \emph{inductive} rather than upper inductive or lower inductive.

Theorem~\ref{thm:BasisExtension} now has the following immediate consequence when coefficients are taken in a finite field.

\begin{corollary}\label{cor:FieldGeneratorBasis}
    Let $G$ be any digraph and $n$ a non-negative integer.
    Then the $n$-dimensional
    inductive elements generate $\Omega_n(G;\mathbb{Z}_p)$. That is,
    \[
        \Omega_n^{I,h}(G;\mathbb{Z}_p) = \Omega_n^{I,t}(G;\mathbb{Z}_p) = \Omega_n(G;\mathbb{Z}_p).
    \]
    Moreover, the $n$-dimensional
    inductive elements contain a subset that is a basis of $\Omega_n(G;\mathbb{Z}_p)$.
\end{corollary}

\begin{proof}
    The corollary in the cases $n=0,1$ follow from Example~\ref{exam:VerticesAndLines}. 
    The full statement follows by induction on $n$ using Theorem~\ref{thm:BasisExtension}.
    \qed
\end{proof}

Recall that when $R$ has no additive torsion other than $2$, face multihypergraphs are face multigraphs.
In particular, when $R=\mathbb{Z}$ or $R=\mathbb{Z}_2$,
inductive elements are obtained by extending over
inductive face multigraphs only.
Theorem~\ref{thm:BasisExtension} also provides the next corollary, which corresponds to Corollary~\ref{cor:FieldGeneratorBasis} in the integral and field of characteristic zero cases.

\begin{corollary}\label{cor:IductiveBasisIntegral}
    Let $G$ be any digraph and $K$ a field of characteristic $0$.
    Then the following statements hold.
    \begin{enumerate}
        \item 
        Inductive elements contain a basis of $\Omega_i(G;\mathbb{Z})$ and $\Omega_i(G;K)$ for $i=0,1,2$.
        \item
        The $3$-dimensional
        inductive elements generate $\Omega_3(G;\mathbb{Z})$.
        \item
        A subset of the $3$-dimensional
        inductive elements forms a basis of $\Omega_3(G;K)$.
    \end{enumerate}
\end{corollary}

\begin{proof}
    We give the proof in the upper inductive case, the lower inductive case being almost identical.
    The corollary in the cases $n=0,1$ follows from Example~\ref{exam:VerticesAndLines} and the determination of bases in Section~\ref{sec:LowDimBasis}.
    In particular, $\bar{E}_0^h$ and $\bar{E}_1^h$ are bases of vertices $e_u$ and edges $e_{u,v}$ respectively for $u,v\in V_G$ and $u\neq v$.
    To obtain the statement of the corollary when $n=2$, we apply Theorem~\ref{thm:BasisExtension} and the result of Section~\ref{sec:LowDimBasis} with the bases $B_{0} = \bar{E}_0^h$ and $B_{1} = \bar{E}_1^h$.
    Together with Example~\ref{exam:DirectedSquare}, this implies that double edges, directed triangles and directed squares are the inductive elements in dimension $2$.
    Hence, using Proposition~\ref{prop:Dim2Base}, we choose a subset of these inductive elements that form a basis $B_2$. Finally, apply again Theorem~\ref{thm:BasisExtension} to show that inductive elements generate $\Omega_3(G;\mathbb{Z})$, and a subset of these forms a basis of $\Omega_3(G;K)$.
    \qed
\end{proof}

Corollaries~\ref{cor:FieldGeneratorBasis}~and~\ref{cor:IductiveBasisIntegral} can together be considered a generalisation of Theorem~\ref{thm:Dim3BasisNoDoubleNoMulti}, 
in the sense that a $\Omega_3(G;R)$ generating set is constructed explicitly without either of the restrictions on double edges or multisquares.
Moreover, using Example~\ref{exam:2Trapezohedron}, it is straightforward to see that trapezohedra of any order are examples of inductive elements.

In the case of integral coefficients, it is not generally true that a generating set of a finite-dimensional free $\mathbb{Z}$-module may be reduced to a basis.
This fact is the only obstruction to further extending Corollary~\ref{cor:IductiveBasisIntegral} to higher dimensions.
However, we are not currently aware of any examples where a subset of inductive elements do not form a basis.
In general, a basis can still be determined from a generating set of $\Omega_n(G;\mathbb{Z})$ by the computation of a Hermite normal form with respect to the intersection with the usual basis of $\mathcal{A}_n(G;\mathbb{Z})$.
In particular, when the number of paths in the generating set $\Omega_n(G;\mathbb{Z})$ is significantly smaller than the rank of $\mathcal{A}_n(G;\mathbb{Z})$, this approach would be more computationally efficient than computing $\Omega_n(G;\mathbb{Z})$ directly as a submodule of $\mathcal{A}_n(G;\mathbb{Z})$.

\section{Important examples}

The examples here are derived and justified using the theory developed in Sections~\ref{sec:StructureMaps},~\ref{sec:Extensions}~and~\ref{sec:InductiveBasis}.
However, once the digraphs have been constructed, 
their properties can be checked directly by computer without reference to the theory developed in this paper.
For example, an algorithm for path homology with real coefficients, including computation of boundary matrices, is publicly available at \cite{Carranza2022}, released accompanying the paper \cite{Carranza2024}. 
Similarly, the details of a straightforward algorithm for path homology with field coefficients can be found in \cite[\S 1.7]{Grigoryan2022}.

Accompanying this work, we provide \cite{Burfitt2024} an implementation
for computing path homology, path homology boundary matrices, and bases of the path chain complex of a digraph with respect to $\mathbb{Q}$ and $\mathbb{Z}_p$ coefficients.

Throughout this section assume that $n$ is a non-negative integer, $G$ is a digraph, $\mathbb{Z}_p$ (for a prime number $p$) is the finite field with $p$ elements, $K$ a field, and $R$ a commutative ring with a unit, unless otherwise stated.

\subsection{Boundary matrix multiplicities}\label{sec:DiffMultiplicityCounter}

The next example, demonstrates that the path homology differential with respect to an inductive basis can contain arbitrary multiplicities in its boundary matrix.

\begin{exmp}\label{exam:ArbitraryMultiplicities}\normalfont
    Let $t\geq 2$ be an integer.
    Throughout this example, all index values are assumed to be integers modulo $2t$.
    We construct a digraph $\mathbb{M}_t$ with vertices
    \[
        V_{\mathbb{M}_t} = 
        \{
        T,
        u^A_1,u^A_2,u^{B},
        v_1^A,v_2^A,v_1^{B},\dots,v_{2t}^{B},
        w^A,w^{B}_1,\dots,w^{B}_{2t},
        H
        \}
    \]
    and edges
    \begin{align*}
        &
        T \to u_1^A, \:
        T \to u_2^A, \:
        T \to u^{B}, 
        \\ &
        u_1^A \to v_1^A, \:
        u_1^A \to v_2^A, \:
        u_2^A \to v_1^A, \:
        u_2^A \to v_2^A, \:
        u^{B} \to v_i^{B}, \:
        u_1^A \to v_{2i+1}^{B}, \:
        u_2^A \to v_{2i}^{B},
        \\ &
        v_1^A \to w^A, \:
        v_2^A \to w^A, \:
        v_i^{B} \to w_i^{B}, \:
        v_i^{B} \to w_{i+1}^{B}, \:
        v_1^A \to w_{2i}^{B}, \:
        v_2^A \to w_{2i+1}^{B},
        \\ &
        w^A \to H, \:
        w^{B}_i \to H
    \end{align*}
    for $i=1,\dots,2t$.
    In the case $t=2$, we obtain the following digraph $\mathbb{M}_2$.
    
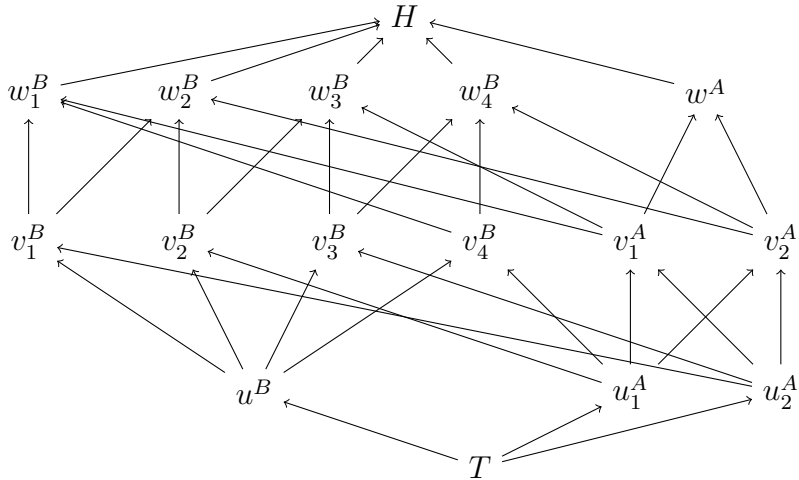
\begin{center}
        \tikz {
            \node (H) at (2,7) {$H$};
            \node (T) at (3,1) {$T$};
            \node (xA) at (5,2) {$u_1^A$};
            \node (x'A) at (7,2) {$u_2^A$};
            \node (x) at (0,2) {$u^{B}$};
            \node (yA) at (5,4) {$v_1^A$};
            \node (y'A) at (7,4) {$v_2^A$};
            \node (yBC) at (-3,4) {$v^{B}_1$};
            \node (yCD) at (-1,4) {$v^{B}_2$};
            \node (yDE) at (1,4) {$v^{B}_3$};
            \node (yBE) at (3,4) {$v^{B}_4$};
            \node (hA) at (6,6) {$w^A$};
            \node (hB) at (-3,6) {$w^{B}_1$};
            \node (hC) at (-1,6) {$w^{B}_2$};
            \node (hD) at (1,6) {$w^{B}_3$};
            \node (hE) at (3,6) {$w^{B}_4$};
            \draw[->] (T) -- (xA);
            \draw[->] (T) -- (x'A);
            \draw[->] (T) -- (x);
            \draw[->] (x) -- (yBC);
            \draw[->] (x) -- (yCD);
            \draw[->] (x) -- (yDE);
            \draw[->] (x) -- (yBE);
            \draw[->] (xA) -- (yA);
            \draw[->] (xA) -- (y'A);
            \draw[->] (xA) -- (yCD);
            \draw[->] (xA) -- (yBE);
            \draw[->] (x'A) -- (yA);
            \draw[->] (x'A) -- (y'A);
            \draw[->] (x'A) -- (yBC);
            \draw[->] (x'A) -- (yDE);
            \draw[->] (yA) -- (hA);
            \draw[->] (yA) -- (hB);
            \draw[->] (yA) -- (hD);
            \draw[->] (y'A) -- (hA);
            \draw[->] (y'A) -- (hC);
            \draw[->] (y'A) -- (hE);
            \draw[->] (yBC) -- (hB);
            \draw[->] (yBC) -- (hC);
            \draw[->] (yCD) -- (hC);
            \draw[->] (yCD) -- (hD);
            \draw[->] (yDE) -- (hD);
            \draw[->] (yDE) -- (hE);
            \draw[->] (yBE) -- (hE);
            \draw[->] (yBE) -- (hB);
            \draw[->] (hA) -- (H);
            \draw[->] (hB) -- (H);
            \draw[->] (hC) -- (H);
            \draw[->] (hD) -- (H);
            \draw[->] (hE) -- (H);
        }
    \end{center}
    All maximal length paths in $\mathbb{M}_{t}$ have length $4$ and $\mathbb{M}_{t}$ contain no directed triangles. It follows that $\Omega_n(\mathbb{M}_{t};\mathbb{Z})=0$ for $n\geq5$.
    
    We now construct an element $I^t_4$ of $\Omega_4(\mathbb{M}_{t};\mathbb{Z})$ as an upper inductive element by identifying the unique (up to mutation equivalence) upper inductive structure it lies over.
    Using this construction, we show that $\partial_4^PI^t_4=t\cdot A+\cdots$ for a certain inductive element $A\in\Omega_3(\mathbb{M}_{t};\mathbb{Z})$. Although the example will be presented with integral coefficients, the conclusions are equally valid for coefficients in a field of characteristic $0$ or $p$, for a prime $p>t$.
    
    First we note that, if the vertex $H$ and all its incoming edges are removed from $\mathbb{M}_{t}$, the remaining subdigraph is the union of the digraphs
    
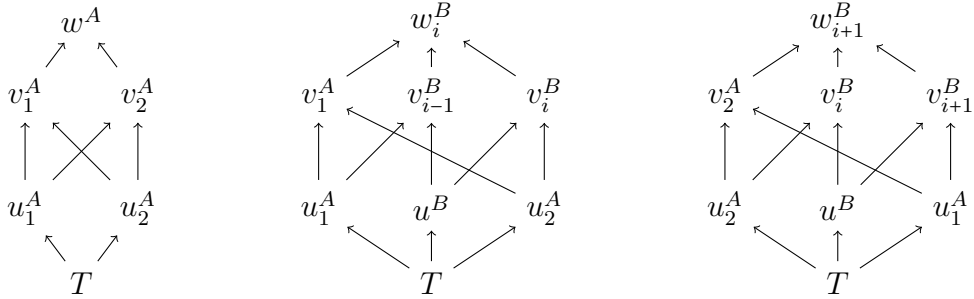
\begin{center}
        \tikz {
            \node (T) at (0,0) {$T$};
            \node (xA) at (-0.75,1) {$u_1^A$};
            \node (x-A) at (0.75,1) {$u_2^A$};
            \node (yA) at (-0.75,2.5) {$v_1^A$};
            \node (y-A) at (0.75,2.5) {$v_2^A$};
            \node (zA) at (0,3.5) {$w^A$};
            \draw[->] (T) -- (xA);
            \draw[->] (T) -- (x-A);
            \draw[->] (xA) -- (yA);
            \draw[->] (xA) -- (y-A);
            \draw[->] (x-A) -- (yA);
            \draw[->] (x-A) -- (y-A);
            \draw[->] (yA) -- (zA);
            \draw[->] (y-A) -- (zA);
        }
        \;\;\;\;\;\;\;\;\;\;\;\;
        \tikz {
            \node (T) at (0,0) {$T$};
            \node (xA) at (-1.5,1) {$u_1^A$};
            \node (x) at (0,1) {$u^{B}$};
            \node (x-A) at (1.5,1) {$u_2^A$};
            \node (yA) at (-1.5,2.5) {$v_1^A$};
            \node (yBE) at (0,2.5) {$v^{B}_{i-1}$};
            \node (yBC) at (1.5,2.5) {$v^{B}_i$};
            \node (zB) at (0,3.5) {$w^{B}_i$};
            \draw[->] (T) -- (xA);
            \draw[->] (T) -- (x);
            \draw[->] (T) -- (x-A);
            \draw[->] (xA) -- (yA);
            \draw[->] (xA) -- (yBE);
            \draw[->] (x) -- (yBC);
            \draw[->] (x) -- (yBE);
            \draw[->] (x-A) -- (yA);
            \draw[->] (x-A) -- (yBC);
            \draw[->] (yA) -- (zB);
            \draw[->] (yBC) -- (zB);
            \draw[->] (yBE) -- (zB);
        }
        \;\;\;\;\;\;\;\;\;\;\;\;
        \tikz {
            \node (T) at (0,0) {$T$};
            \node (x-A) at (-1.5,1) {$u_2^A$};
            \node (x) at (0,1) {$u^{B}$};
            \node (xA) at (1.5,1) {$u_1^A$};
            \node (y-A) at (-1.5,2.5) {$v_2^A$};
            \node (yBC) at (0,2.5) {$v_i^{B}$};
            \node (yCD) at (1.5,2.5) {$v_{i+1}^{B}$};
            \node (zC) at (0,3.5) {$w^{B}_{i+1}$};
            \draw[->] (T) -- (xA);
            \draw[->] (T) -- (x);
            \draw[->] (T) -- (x-A);
            \draw[->] (x-A) -- (y-A);
            \draw[->] (x-A) -- (yBC);
            \draw[->] (x) -- (yCD);
            \draw[->] (x) -- (yBC);
            \draw[->] (xA) -- (yA);
            \draw[->] (xA) -- (yCD);
            \draw[->] (y-A) -- (zC);
            \draw[->] (yCD) -- (zC);
            \draw[->] (yBC) -- (zC);
        }
    \end{center}
    where $i=1,3\dots,2t-1$ is an odd integer.
    In particular, by Definition~\ref{def:Trapezohedron}, the digraphs above are trapezohedra of order $2$, $3$, and $3$, respectively.
    Each of the trapezohedron digraphs contains a unique trapezohedron element generating the dimension $3$ path chains
    \begin{align*}
        A & =
        - e_{T,u_1^A,v_1^A,w_A} +
        e_{T,u_1^A,v_2^A,w_A} -
        e_{T,u_2^A,v_2^A,w_A} +
        e_{T,u_2^A,v_1^A,w_A}
        \\
        B_i & =
        e_{T,u_1^A,v_1^A,w_i^B} -
        e_{T,u_1^A,v_{i-1}^B,w_i^B} +
        e_{T,u^B,v_{i-1}^B,w_i^B} -
        e_{T,u^B,v_i^B,w_i^B} +
        e_{T,u_2^A,v_i^B,w_i^B} -
        e_{T,u_2^A,v_1^A,w_i^B}
        \\
        B_{i+1} & =
        e_{T,u_2^A,v_2^A,w_{i+1}^B} -
        e_{T,u_2^A,v_i^B,w_{i+1}^B} +
        e_{T,u^B,v_i^B,w_{i+1}^B} -
        e_{T,u^B,v_{i+1}^B,w_{i+1}^B} +
        e_{T,u_1^A,v_{i+1}^B,w_{i+1}^B} -
        e_{T,u_1^A,v_2^A,w_{i+1}^B}
    \end{align*}
    respectively, where $i=1,3\dots,2t-1$.
    We can verify the above generators by applying Theorem~\ref{thm:BasisExtension} and realising $A$, $B_i$, $B_{i+1}$ for $i=1,3\dots,2t-1$ as upper extensions
    \[
        A = [-S^A_1-S^A_2]^{w^a}, \;\;\;
        B_i = [S^{B}_i-S^{B}_{i-1}+S^A_1]^{w_i^B}, \;\;\;
        B_{i+1} = [S^{B}_{i+1}-S^{B}_i+S^A_2]^{w^B_{i+1}},
    \]
    over the face multigraphs,
    
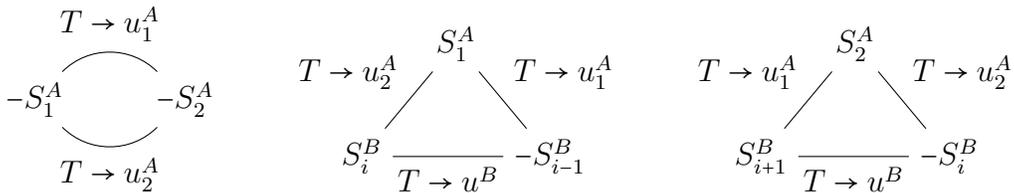
\begin{center}
        \tikz {
            \node (S1A) at (0,0) {$-S_1^A$};
            \node (S2A) at (2,0) {$-S_2^A$};
            \draw[-] (S1A) to [out=45,in=135] node[pos=0.5,above] {$T\to u^A_1$} (S2A);
            \draw[-] (S1A) to [out=315,in=225] node[pos=0.5,below] {$T\to u^A_2$} (S2A);
        }
        \;\;\;\;\;
        \tikz {
            \node (A) at (0,0) {$S^A_1$};
            \node (BC) at (-1.25,-1.5) {$S^{B}_i$};
            \node (BE) at (1.25,-1.5) {$-S^{B}_{i-1}$};
            \draw[-] (BC) -- (BE) node[midway,below] {$T \to u^B$};
            \draw[-] (A) -- (BC) node[midway,above left] {$T \to u^A_2$};
            \draw[-] (A) -- (BE) node[midway,above right] {$T \to u^A_1$};
        }
        \;\;\;\;\;
        \tikz {
            \node (A) at (0,0) {$S^A_2$};
            \node (CD) at (-1.25,-1.5) {$S^{B}_{i+1}$};
            \node (BC) at (1.25,-1.5) {$-S^{B}_i$};
            \draw[-] (CD) -- (BC) node[midway,below] {$T \to u^B$};
            \draw[-] (A) -- (CD) node[midway,above left] {$T \to u^A_1$};
            \draw[-] (A) -- (BC) node[midway,above right] {$T \to u^A_2$};
        }
    \end{center}
    where
    \begin{align*}
        S_1^A &=
        e_{T,u^A_2,v^A_1}
        -
        e_{T,u^A_1,v^A_1}
        , \;\;\;
        S_2^A =
        e_{T,u^A_1,v^A_2}
        -
        e_{T,u^A_2,v^A_2}
        ,\\
        S^{B}_i &=
        e_{T,u^B,v^B_i}
        -
        e_{T,u^A_2,v^B_i}
        , \;\;\;
        S^{B}_{i+1} =
        e_{T,u^B,v^B_{i+1}}
        -
        e_{T,u^A_1,v^B_{i+1}}
    \end{align*}
    are a basis of directed squares whose tail vertices are $T$.
    In fact, it can also be easily checked that the above generators form a basis of
    \[
        \Omega_3^{T,w^A}(\mathbb{M}_{t};\mathbb{Z}),
        \;\;\;
        \Omega_3^{T,w^B_i}(\mathbb{M}_{t};\mathbb{Z})
        \;\;\; \text{and} \;\;\;
        \Omega_3^{T,w^B_{i+1}}(\mathbb{M}_{t};\mathbb{Z})
    \]
    respectively for $i=2,4,\dots,2t$.
    Moreover, together the above elements form a basis of $\Omega_3^{T,*}(\mathbb{M}_{t};\mathbb{Z})$.
    
    We obtain an element of $\Omega_4(\mathbb{M}_t;\mathbb{Z})$ as the upper extension
    \[
        I_4^t = [\underbrace{A+\cdots+A}_{t}+B_1+\dots+B_{2t}]^H
    \]
    which is an extension over the following strongly connected $H$-complete upper face multigraph.
    \begin{center}
        \tikz {
            \node (B1) at (0,4) {$B_1$};
            \node (B2) at (2,6) {$B_2$};
            \node (B3) at (4,6) {$B_3$};
            \node (B4) at (6,4) {$B_4$};
            \node (Bt-1) at (2,0) {$B_{2t-1}$};
            \node (Bt) at (0,2) {$B_{2t}$};
            \node (A2) at (2,4) {$A$};
            \node (A4) at (4,4) {$A$};
            \node (A2t) at (2,2) {$A$};
            \node (D1) at (4,2) {$\udots$};
            \node (D2) at (5,1) {$\udots$};
            \node (E1) at (6,3) {};
            \node (E2) at (3.25,0) {};
            \draw[-] (B1) -- (B2) node[midway,above left] {$S^B_{1}$};
            \draw[-] (B2) -- (B3) node[midway,above] {$S^B_{2}$};
            \draw[-] (B3) -- (B4) node[midway,above right] {$S^B_{3}$};
            \draw[-] (Bt-1) -- (Bt) node[midway,below left] {$S^B_{2t-1}$};
            \draw[-] (Bt) -- (B1) node[midway,left] {$S^B_{2t}$};
            \draw[-] (A2) -- (B1) node[midway,below] {$S^A_{1}$};
            \draw[-] (A2) -- (B2) node[midway,right] {$S^A_{2}$};
            \draw[-] (A4) -- (B3) node[midway,left] {$S^A_{1}$};
            \draw[-] (A4) -- (B4) node[midway,below] {$S^A_{2}$};
            \draw[-] (A2t) -- (Bt-1) node[midway,right] {$S^A_{1}$};
            \draw[-] (A2t) -- (Bt) node[midway,above] {$S^A_{2}$};
            \draw[-] (B4) -- (E1);
            \draw[-] (E2) -- (Bt-1);
        }
    \end{center}
    However, the upper face multigraph above is not unique up to mutation equivalence.  
    For example, when $t=2$ either of the following two mutation equivalent face multigraphs would suffice as an upper inductive structure.
    \begin{center}
        \tikz {
            \node (A1) at (-0.25,3) {$A$};
            \node (A2) at (0.25,2) {$A$};
            \node (B) at (-2.5,2.5) {$B_1$};
            \node (C) at (0,5) {$B_2$};
            \node (D) at (2.5,2.5) {$B_3$};
            \node (E) at (0,0) {$B_4$};
            \draw[-] (B) -- (C) node[midway,above left] {$S^B_{1}$};
            \draw[-] (C) -- (D) node[midway,above right] {$S^B_{2}$};
            \draw[-] (D) -- (E) node[midway,below right] {$S^B_{3}$};
            \draw[-] (E) -- (B) node[midway,below left] {$S^B_{4}$};
            \draw[-] (B) -- (A1) node[midway,below] {$S^A_1$};
            \draw[-] (C) -- (A1) node[midway,right] {$S^A_2$};
            \draw[-] (D) -- (A2) node[midway,above] {$S^A_1$};
            \draw[-] (E) -- (A2) node[midway,left] {$S^A_2$};
        }
        \;\;\;\;\;\;\;\;\;
        \tikz {
            \node (A1) at (-0.25,2) {$A$};
            \node (A2) at (0.25,3) {$A$};
            \node (B) at (-2.5,2.5) {$B_1$};
            \node (C) at (0,5) {$B_2$};
            \node (D) at (2.5,2.5) {$B_3$};
            \node (E) at (0,0) {$B_4$};
            \draw[-] (B) -- (C) node[midway,above left] {$S^B_{1}$};
            \draw[-] (C) -- (D) node[midway,above right] {$S^B_{2}$};
            \draw[-] (D) -- (E) node[midway,below right] {$S^B_{3}$};
            \draw[-] (E) -- (B) node[midway,below left] {$S^B_{4}$};
            \draw[-] (B) -- (A1) node[midway,above] {$S^A_1$};
            \draw[-] (E) -- (A1) node[midway,right] {$S^A_2$};
            \draw[-] (C) -- (A2) node[midway,left] {$S^A_2$};
            \draw[-] (D) -- (A2) node[midway,below] {$S^A_1$};
        }
    \end{center}
    
    Finally, due to the structure of the $\Omega_3^{T,*}(\mathbb{M}_{t};\mathbb{Z})$ basis, the inductive structure on $I_4^t$ above is the only strongly connected $H$-complete face multigraph that can be constructed on lower connected elements of $\Omega_3(\mathbb{M}_t;\mathbb{Z})$ with tail $T$ up to mutations.
    Therefore, by Theorem~\ref{thm:BasisExtension}, $I_4^t$ is the unique generator of $\Omega_4(\mathbb{M}_t;\mathbb{Z})$ up to sign.
    Furthermore, the element $I_4^t$ is constructed using an inductive structure containing $t$ copies of the element $A\in \Omega_3(\mathbb{M}_t;\mathbb{Z})$. 
    Hence, with respect to a $\Omega_*(\mathbb{M}_t;\mathbb{Z})$ basis extending the elements used to construct $I_4^t$, the image of $\delta^h_4$ consists of a vector containing an element of degree $t$. 
    Moreover, using the $\Omega_{*}^{*,*}(G;R)$ bigrading detailed in equation~\eqref{eq:Bigrading}, we see that the image of $\delta^h_4$ lies in a separate summand of $\Omega_3(\mathbb{M}_t;\mathbb{Z})$ from $\partial^P_{4}-\delta^h_4$.
    Therefore, the matrix representing $\partial^P_4$ with respect to the $\Omega_*(\mathbb{M}_t;\mathbb{Z})$ basis chosen above,
    has at least one entry of multiplicity $t$.
\end{exmp}

\subsection{Digraphs with different \texorpdfstring{$\mathbb{Z}_p$}{ finite fields}
Euler characteristics}\label{sec:DiffEulerOverFiniteFields}

The following definition is from \cite[\S 5.2]{Fu2024}.
Let $K$ be a field and let $G$ be a finite digraph such that only finitely many $H^{P}_{n}(G;K) \neq 0$ for $n\geq 0$.
In this case, the \emph{(path) Euler characteristic} of $G$ with coefficients $K$ is given by
\[
    \chi^K(G)
    = \sum_{i=0}^{\infty} (-1)^i\dim(H^{P}_{i}(G;K))
    = \sum_{i=0}^{\infty} (-1)^i\dim(\Omega_{i}(G;K)).
\]
The next example is, to our knowledge, the first known construction of a digraph $G$ such that $\chi^\mathbb{Q}(G) \neq \chi^{\mathbb{Z}_p}(G)$ when $p \geq 3$.
In particular, the example answers the second open question of Fu and Ivanov set out in the introduction of \cite{Fu2024}.

In the course of their work, Fu and Ivanov show that a digraph $G$ containing no multisquares satisfies $\chi^\mathbb{Q}(G)=\chi^{\mathbb{Z}_p}(G)$ for any odd prime $p$. In the following example, it is therefore essential that the construction of inductive elements is applicable to digraphs containing multisquares.

\begin{exmp}\label{exam:DiffEulerOverFiniteFields}\normalfont
    Let $t \geq 2$ be an integer.
    Throughout the example, all index values are assumed to be integers modulo $t$.
    We construct a digraph $\mathbb{E}_t$ with vertices
    \[
        V_{\mathbb{E}_t} = 
        \{
        T,
        u^A_1,u^A_2,
        u_1^C,\dots,u_{t}^C,
        v^A,
        v_1^{B_1},\dots,v_{t}^{B_1},
        v_1^{B_2},\dots,v_{t}^{B_2},
        v_1^C,\dots,v_{t}^C,
        w_1,\dots,w_{t},
        H
        \}
    \]
    and edges
    \begin{align*}
        &
        T \to u_1^A, \:
        T \to u_2^A, \:
        T \to u_i^C, \:
        T \to v_i^C,
        \\ &
        u_1^A \to v^A, \:
        u_2^A \to v^A, \:
        u_1^A \to v_{i}^{B_1}, \:
        u_2^A \to v_{i}^{B_2}, \:
        u_i^C \to v_i^{B_1}, \:
        u_i^C \to v_i^{B_2}, \:
        u_i^C \to v_i^C, \:
        \\ &
        v^A \to w_i, \:
        v_i^{B_1} \to w_{i+1}, \:
        v_i^{B_2} \to w_i, \:
        v_{i}^C \to w_i, \:
        v_{i}^C \to w_{i+1},
        \\ &
        v^{B_1}_i \to H, \: v^{B_2}_i \to H, w_i \to H
    \end{align*}
    for $i=1,\dots,t$.
    In the case $t=3$, we obtain the following digraph $\mathbb{E}_3$.
    
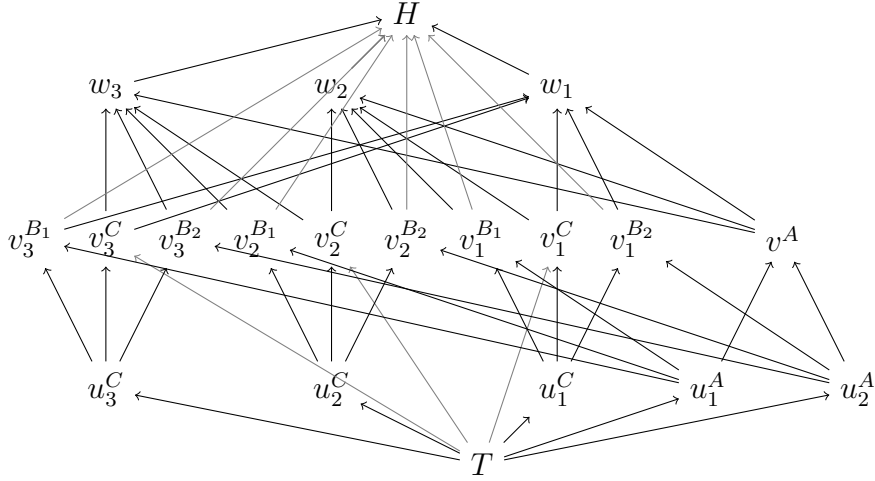
\begin{center}
        \tikz {
            \node (H) at (-3,6) {$H$};
            \node (T) at (-2,0) {$T$};
            \node (uA1) at (1,1) {$u_1^A$};
            \node (uA2) at (3,1) {$u_2^A$};
            \node (uC1) at (-1,1) {$u_1^C$};
            \node (uC2) at (-4,1) {$u_2^C$};
            \node (uC3) at (-7,1) {$u_3^C$};
            \node (vA) at (2,3) {$v^A$};
            \node (vC1) at (-1,3) {$v_1^C$};
            \node (vC2) at (-4,3) {$v_2^C$};
            \node (vC3) at (-7,3) {$v_3^C$};
            \node (v1B1) at (-2,3) {$v_1^{B_1}$};
            \node (v1B2) at (-5,3) {$v_2^{B_1}$};
            \node (v1B3) at (-8,3) {$v_3^{B_1}$};
            \node (v2B1) at (0,3) {$v_1^{B_2}$};
            \node (v2B2) at (-3,3) {$v_2^{B_2}$};
            \node (v2B3) at (-6,3) {$v_3^{B_2}$};
            \node (w1) at (-1,5) {$w_1$};
            \node (w2) at (-4,5) {$w_2$};
            \node (w3) at (-7,5) {$w_3$};
            \draw[->] (T) -- (uA1);
            \draw[->] (T) -- (uA2);
            \draw[->] (T) -- (uC1);
            \draw[->] (T) -- (uC2);
            \draw[->] (T) -- (uC3);
            \draw[->,gray] (T) -- (vC1);
            \draw[->,gray] (T) -- (vC2);
            \draw[->,gray] (T) -- (vC3);
            \draw[->] (uA1) -- (vA);
            \draw[->] (uA2) -- (vA);
            \draw[->] (uC1) -- (vC1);
            \draw[->] (uC2) -- (vC2);
            \draw[->] (uC3) -- (vC3);
            \draw[->] (uC1) -- (v1B1);
            \draw[->] (uC2) -- (v1B2);
            \draw[->] (uC3) -- (v1B3);
            \draw[->] (uC1) -- (v2B1);
            \draw[->] (uC2) -- (v2B2);
            \draw[->] (uC3) -- (v2B3);
            \draw[->] (uA1) -- (v1B1);
            \draw[->] (uA1) -- (v1B2);
            \draw[->] (uA1) -- (v1B3);
            \draw[->] (uA2) -- (v2B1);
            \draw[->] (uA2) -- (v2B2);
            \draw[->] (uA2) -- (v2B3);
            \draw[->] (vA) -- (w1);
            \draw[->] (vA) -- (w2);
            \draw[->] (vA) -- (w3);
            \draw[->] (vC1) -- (w1);
            \draw[->] (vC2) -- (w2);
            \draw[->] (vC3) -- (w3);
            \draw[->] (vC1) -- (w2);
            \draw[->] (vC2) -- (w3);
            \draw[->] (vC3) -- (w1);
            \draw[->] (v1B1) -- (w2);
            \draw[->] (v1B2) -- (w3);
            \draw[->] (v1B3) -- (w1);
            \draw[->] (v2B1) -- (w1);
            \draw[->] (v2B2) -- (w2);
            \draw[->] (v2B3) -- (w3);
            \draw[->] (w1) -- (H);
            \draw[->] (w2) -- (H);
            \draw[->] (w3) -- (H);
            \draw[->,gray] (v1B1) -- (H);
            \draw[->,gray] (v1B2) -- (H);
            \draw[->,gray] (v1B3) -- (H);
            \draw[->,gray] (v2B1) -- (H);
            \draw[->,gray] (v2B2) -- (H);
            \draw[->,gray] (v2B3) -- (H);
        }
    \end{center}
    All maximal length paths in $\mathbb{E}_{t}$ have length $3$ or $4$. Therefore $\Omega_n(G;R) = 0$ when $n\geq 5$.

    When $t$ is prime, we now construct an element of $\Omega_4(\mathbb{E}_{t};\mathbb{Z}_t)$ as an upper inductive element by identifying the unique up to sign upper inductive structure it lies over.
    We first note that, if the vertex $H$ and all its incoming edges are removed from $\mathbb{E}_{t}$, the remaining subdigraph is the union of the digraphs
    
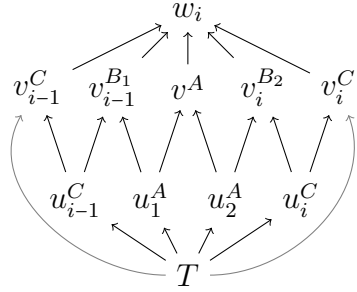
\begin{center}
        \tikz {
            \node (T) at (0,0) {$T$};
            \node (uCi) at (1.5,1) {$u_i^C$};
            \node (uA1) at (-0.5,1) {$u_1^A$};
            \node (uA2) at (0.5,1) {$u_2^A$};
            \node (uCi-1) at (-1.5,1) {$u_{i-1}^C$};
            \node (vA) at (0,2.5) {$v^A$};
            \node (vB1i-1) at (-1,2.5) {$v_{i-1}^{B_1}$};
            \node (vB2i) at (1,2.5) {$v_i^{B_2}$};
            \node (vCi-1) at (-2,2.5) {$v_{i-1}^C$};
            \node (vCi) at (2,2.5) {$v_i^C$};
            \node (w) at (0,3.5) {$w_i$};
            \draw[->] (T) -- (uCi-1);
            \draw[->] (T) -- (uA1);
            \draw[->] (T) -- (uA2);
            \draw[->] (T) -- (uCi);
            \draw[->] (uCi-1) -- (vCi-1);
            \draw[->] (uCi-1) -- (vB1i-1);
            \draw[->] (uA1) -- (vB1i-1);
            \draw[->] (uA1) -- (vA);
            \draw[->] (uA2) -- (vA);
            \draw[->] (uA2) -- (vB2i);
            \draw[->] (uCi) -- (vB2i);
            \draw[->] (uCi) -- (vCi);
            \draw[->] (vCi-1) -- (w);
            \draw[->] (vB1i-1) -- (w);
            \draw[->] (vA) -- (w);
            \draw[->] (vB2i) -- (w);
            \draw[->] (vCi) -- (w);
            \draw[->,gray] (T) to [out=180,in=240] (vCi-1);
            \draw[->,gray] (T) to [out=0,in=290] (vCi);
        }
    \end{center}
    for $i=1\dots,t$.
    Each of the digraphs above have up to sign a single generator of their dimension $3$ path chains, which can be constructed as an upper inductive element as follows.
    
    Consider the $\mathbb{E}_t$ directed squares and directed triangles
    \begin{align*}
        S^A &= e_{T,u_1^A,v^A} - e_{T,u_2^A,v^A} \\
        S^{B_1}_i &= e_{T,u_{i}^C,v_{i}^{B_1}} - e_{T,u_1^A,v_{i}^{B_1}} \\
        S^{B_2}_i &= e_{T,u_2^A,v_{i}^{B_2}} - e_{T,u_{i}^C,v_{i}^{B_2}} \\
        T_i^C &= e_{T,u_{i}^C,v_{i}^C}
    \end{align*}
    which by Proposition~\ref{prop:Dim2Base} are up sign the only generators of $\Omega_2^{T,v}(\mathbb{E}_t;R)$ for any $v\in V_G$.
    The generators of the dimension $3$ path chains of the digraphs above are given by the upper extensions
    \[
        E_i = [-T_{i-1}^C+S^{B_1}_{i-1}+S^A+S^{B_2}_i+T_i^C]^{w_i} 
    \]
    respectively where $i=1\dots,t$, over the strongly connected $w_i$-complete face multigraphs
    \begin{center}
        \tikz {
            \node (1) at (0,0) {$-T^C_{i-1}$};
            \node (2) at (3,0) {$S^{B_1}_{i-1}$};
            \node (3) at (6,0) {$S^A$};
            \node (4) at (9,0) {$S^{B_2}_i$};
            \node (5) at (12,0) {$T_i^C.$};
            \draw[-] (1) -- (2) node[midway,above] {$T \to u^C_{i-1}$};
            \draw[-] (2) -- (3) node[midway,above] {$T \to u^A_1$};
            \draw[-] (3) -- (4) node[midway,above] {$T \to u^A_2$};
            \draw[-] (4) -- (5) node[midway,above] {$T \to u^C_{i}$};
            }
    \end{center}
    In particular, as no other face multihypergraphs can be constructed up to sign, the only generators up to sign of $\Omega_3^{T,w_i}(G;R)$ are those obtained using the extensions above.
    Moreover, the only face multihypergraph up to sign and mutation that can be constructed on the elements $E_i$ for $i=1,\dots,2t$ is
    
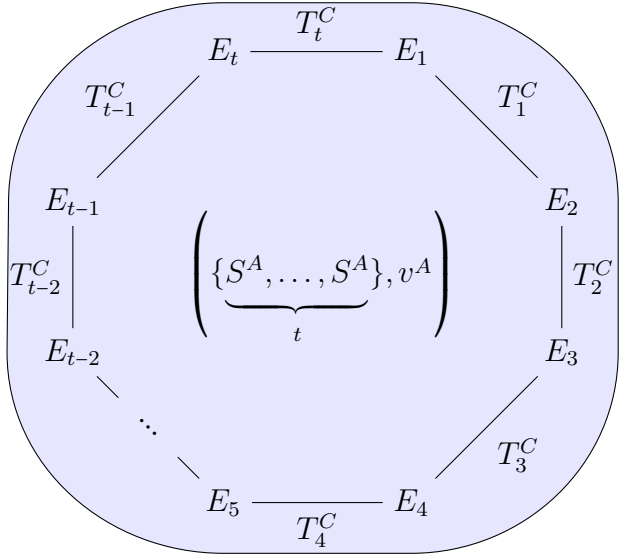
\begin{center}
        \begin{tikzpicture}
            \node (IE1) at (4.5,5) {};
            \node (IE2) at (6.5,3) {};
            \node (IEt-1) at (6.5,1) {};
            \node (IEt) at (4.5,-1) {};
            \node (IEt+1) at (2,-1) {};
            \node (IEt+2) at (0,1) {};
            \node (IE2t-1) at (0,3) {};
            \node (IE2t) at (2,5) {};
            \filldraw[fill=blue!10] ($(IE1)+(0,0.65)$)
            to[out=0,in=90] ($(IE2)+(0.75,0)$)
            -- ($(IEt-1)+(0.75,0)$)
            to[out=270,in=0] ($(IEt)+(0,-0.675)$)
            -- ($(IEt+1)+(0,-0.675)$)
            to[out=180,in=270] ($(IEt+2)+(-0.875,0)$)
            -- ($(IE2t-1)+(-0.855,0)$)
            to[out=90,in=180] ($(IE2t)+(0,0.65)$)
            -- ($(IE1)+(0,0.65)$);
            \node (SA1) at (3.3,2) {$\left(\{\underbrace{S^A,\dots,S^A}_{t}\},v^A\right)$};
            \node (E1) at (4.5,5) {$E_1$};
            \node (E2) at (6.5,3) {$E_2$};
            \node (Et-1) at (6.5,1) {$E_3$};
            \node (Et) at (4.5,-1) {$E_4$};
            \node (Et+1) at (2,-1) {$E_5$};
            \node (Et+2) at (0,1) {$E_{t-2}$};
            \node (E2t-1) at (0,3) {$E_{t-1}$};
            \node (E2t) at (2,5) {$E_{t}$};
            \node (D) at (1,0) {$\ddots$};
            \node (Et+1-) at (1.25,-0.25) {};
            \node (Et+2-) at (0.75,0.25) {};
            \draw[-] (E1) -- (E2) node[midway,above right] {$T^C_1$};
            \draw[-] (E2) -- (Et-1) node[midway,right] {$T^C_2$};
            \draw[-] (Et-1) -- (Et) node[midway,below right] {$T^C_{3}$};
            \draw[-] (Et) -- (Et+1) node[midway,below] {$T^C_4$};
            \draw[-] (Et+2) -- (E2t-1) node[midway,left] {$T^C_{t-2}$};
            \draw[-] (E2t-1) -- (E2t) node[midway,above left] {$T^C_{t-1}$};
            \draw[-] (E2t) -- (E1) node[midway,above] {$T^C_{t}$};
            \draw[-] (Et+1) -- (Et+1-);
            \draw[-] (Et+2-) -- (Et+2);
        \end{tikzpicture}
    \end{center}
    which is an inductive structure for the upper extension
    \[
        [E_1+\cdots+E_{t}]^H.
    \]
    The face multihypergraph above cannot be constructed with coefficients $\mathbb{Z}$, $\mathbb{Q}$ or $\mathbb{Z}_p$ for prime $p\neq t$, as there is no combination of labeled hyperedges that contains all of the $t$ required $S^A$ labels simultaneously unless $p$ divides $t$, which is impossible as $p$ and $t$ are distinct primes.
    
    In Examples~\ref{exam:VerticesAndLines}~and~\ref{exam:DirectedSquare}, a basis for $\Omega_i(G;R)$ when $i=0,1,2$ is constructed in the same way for all coefficient rings.
    All these basis elements are inductive elements and the basis has the same size irrespective of the choice of coefficients.
    In the case $i=3$, similarly to the argument above for the elements $E_i$, we can deduce that $\Omega_3(\mathbb{E}_t;R)$ has a basis consisting of $E_i$ and the unique up to sign generators in the dimension $3$ path chains obtained from $\mathbb{E}_t$ subdigraphs
    
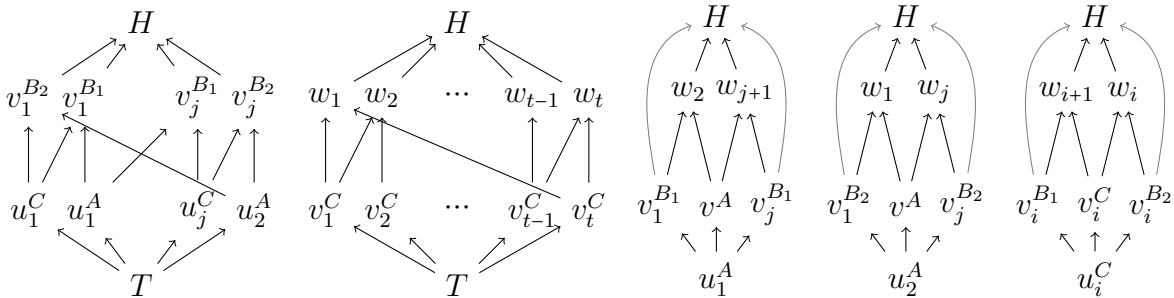
\begin{center}
        \tikz {
            \node (T) at (0,0) {$T$};
            \node (u1) at (-1.5,1) {$u_1^C$};
            \node (u2) at (-0.75,1) {$u_1^A$};
            \node (u3) at (0.75,1) {$u_j^C$};
            \node (u4) at (1.5,1) {$u_2^A$};
            \node (v1) at (-1.5,2.5) {$v_1^{B_2}$};
            \node (v2) at (-0.75,2.5) {$v_1^{B_1}$};
            \node (v3) at (0.75,2.5) {$v_j^{B_1}$};
            \node (v4) at (1.5,2.5) {$v_j^{B_2}$};
            \node (H) at (0,3.5) {$H$};
            \draw[->] (T) -- (u1);
            \draw[->] (T) -- (u2);
            \draw[->] (T) -- (u3);
            \draw[->] (T) -- (u4);
            \draw[->] (u1) -- (v1);
            \draw[->] (u2) -- (v2);
            \draw[->] (u3) -- (v3);
            \draw[->] (u4) -- (v4);
            \draw[->] (u1) -- (v2);
            \draw[->] (u2) -- (v3);
            \draw[->] (u3) -- (v4);
            \draw[->] (u4) -- (v1);
            \draw[->] (v1) -- (H);
            \draw[->] (v2) -- (H);
            \draw[->] (v3) -- (H);
            \draw[->] (v4) -- (H);
        }
        \tikz {
            \node (T) at (0,0) {$T$};
            \node (u1) at (-1.75,1) {$v_1^C$};
            \node (u2) at (-1,1) {$v_2^C$};
            \node (u3) at (1,1) {$v_{t-1}^C$};
            \node (u4) at (1.75,1) {$v_t^C$};
            \node (ud) at (0,1) {$\cdots$};
            \node (v1) at (-1.75,2.5) {$w_1$};
            \node (v2) at (-1,2.5) {$w_2$};
            \node (v3) at (1,2.5) {$w_{t-1}$};
            \node (v4) at (1.75,2.5) {$w_t$};
            \node (vd) at (0,2.5) {$\cdots$};
            \node (H) at (0,3.5) {$H$};
            \draw[->] (T) -- (u1);
            \draw[->] (T) -- (u2);
            \draw[->] (T) -- (u3);
            \draw[->] (T) -- (u4);
            \draw[->] (u1) -- (v1);
            \draw[->] (u2) -- (v2);
            \draw[->] (u3) -- (v3);
            \draw[->] (u4) -- (v4);
            \draw[->] (u1) -- (v2);
            \draw[->] (u3) -- (v4);
            \draw[->] (u4) -- (v1);
            \draw[->] (v1) -- (H);
            \draw[->] (v2) -- (H);
            \draw[->] (v3) -- (H);
            \draw[->] (v4) -- (H);
        }
        \tikz {
            \node (T) at (0,0) {$u_1^A$};
            \node (u1) at (-0.75,1) {$v_1^{B_1}$};
            \node (u2) at (0,1) {$v^A$};
            \node (u3) at (0.75,1) {$v_j^{B_1}$};
            \node (v1) at (-0.375,2.5) {$w_2$};
            \node (v2) at (0.375,2.5) {$w_{j+1}$};
            \node (H) at (0,3.5) {$H$};
            \draw[->] (T) -- (u1);
            \draw[->] (T) -- (u2);
            \draw[->] (T) -- (u3);
            \draw[->] (u1) -- (v1);
            \draw[->] (u2) -- (v1);
            \draw[->] (u2) -- (v2);
            \draw[->] (u3) -- (v2);
            \draw[->] (v1) -- (H);
            \draw[->] (v2) -- (H);
            \draw[->,gray] (u1) to [out=100,in=200] (H);
            \draw[->,gray] (u3) to [out=80,in=340] (H);
        }
        \tikz {
            \node (T) at (0,0) {$u_2^A$};
            \node (u1) at (-0.75,1) {$v_1^{B_2}$};
            \node (u2) at (0,1) {$v^A$};
            \node (u3) at (0.75,1) {$v_j^{B_2}$};
            \node (v1) at (-0.375,2.5) {$w_1$};
            \node (v2) at (0.375,2.5) {$w_{j}$};
            \node (H) at (0,3.5) {$H$};
            \draw[->] (T) -- (u1);
            \draw[->] (T) -- (u2);
            \draw[->] (T) -- (u3);
            \draw[->] (u1) -- (v1);
            \draw[->] (u2) -- (v1);
            \draw[->] (u2) -- (v2);
            \draw[->] (u3) -- (v2);
            \draw[->] (v1) -- (H);
            \draw[->] (v2) -- (H);
            \draw[->,gray] (u1) to [out=100,in=200] (H);
            \draw[->,gray] (u3) to [out=80,in=340] (H);
        }
        \tikz {
            \node (T) at (0,0) {$u_i^C$};
            \node (u1) at (-0.75,1) {$v_i^{B_1}$};
            \node (u2) at (0,1) {$v_i^C$};
            \node (u3) at (0.75,1) {$v_i^{B_2}$};
            \node (v1) at (-0.375,2.5) {$w_{i+1}$};
            \node (v2) at (0.375,2.5) {$w_i$};
            \node (H) at (0,3.5) {$H$};
            \draw[->] (T) -- (u1);
            \draw[->] (T) -- (u2);
            \draw[->] (T) -- (u3);
            \draw[->] (u1) -- (v1);
            \draw[->] (u2) -- (v1);
            \draw[->] (u2) -- (v2);
            \draw[->] (u3) -- (v2);
            \draw[->] (v1) -- (H);
            \draw[->] (v2) -- (H);
            \draw[->,gray] (u1) to [out=100,in=200] (H);
            \draw[->,gray] (u3) to [out=80,in=340] (H);
        }
    \end{center}
    where $i=1,\dots,t$ and $j=2,\dots,t$.
    In particular, the above basis generators exist over any field as they can be obtained as upper extensions over upper inductive face multigraphs similarly to elements $E_i$.
    Moreover, by counting the number of basis elements, we deduce that $\dim(\Omega_{3}(\mathbb{E}_t;K)) = 2t+3(t-1)+1=5t-2$ for any field $K$.
    
    Finally we conclude that, as the dimension of $\Omega_{i}(\mathbb{E}_t;K)$ depends only on the choice of field $K$ when $i=4$, the path Euler characteristic of $\mathbb{E}_t$ with coefficients $\mathbb{Z}_t$ for prime $t$ differs from that of $\mathbb{E}_t$ with coefficients $\mathbb{Q}$ and $\mathbb{Z}_p$ when $p \neq t$.

    More generally, let $t\geq 2$ be an arbitrary integer with unique prime factorisation
    \[
        t = p_1^{i_1}\cdots p_{k}^{i_k}
    \]
    where $p_1,\dots,p_k$ are distinct primes and $i_1,\dots,i_k$ positive integers.
    Similarly to the argument above, we obtain that
    \[
        \dim(\Omega_{4}(\mathbb{E}_{t};K)) =
        \begin{cases}
            1 & \text{if} \; K = \mathbb{Z}_{p_j} \: \text{for} \: j \in \{ 1,\dots,k\},
            \\
            0 & \text{otherwise} 
        \end{cases}
    \]
    and
    \[
        \dim(\Omega_i(\mathbb{E}_{t};K)) = \dim(\Omega_i(\mathbb{E}_{t};K'))
    \]
    for coefficient fields $K$, $K'$ with any integer $i \neq 4$.

\end{exmp}

\bibliographystyle{amsplain}
\bibliography{References}

\end{document}